\documentclass[11pt]{amsart}
\usepackage{fullpage}
\usepackage[titletoc]{appendix}

\newif\ifHideFoot
\HideFoottrue 
\HideFootfalse 

\usepackage{amssymb,amsmath,amsthm,amscd,mathrsfs,graphicx, color}
\usepackage[color,cmtip,all,matrix,arrow,tips,curve]{xy}
\usepackage[colorlinks=true]{hyperref}
\usepackage{multicol}
\usepackage{tikz}
\usepackage{tikz-cd}

\numberwithin{equation}{section}
\newtheorem{teo}{Theorem}[section]
\newtheorem{pro}[teo]{Proposition}
\newtheorem{lem}[teo]{Lemma}
\newtheorem{cor}[teo]{Corollary}

\newtheorem{teoalpha}{Theorem}

\newtheorem{coralpha}[teoalpha]{Corollary}

\theoremstyle{definition}

\theoremstyle{remark}
\newtheorem{rem}[teo]{Remark}

\newcommand{\marg}[1]{\normalsize{{\color{red}\footnote{{\color{blue}#1}}}{\marginpar[{\vskip
    -.3 cm\color{red}\hfill\tiny\thefootnote$\rightarrow$}]{{\vskip -.3 cm
     \color{red}$\leftarrow$\tiny\thefootnote}}}}}
\newcommand{\Yano}[1]{\marg{(Yano) #1}}
\newcommand{\Jeff}[1]{\marg{(Jeff) #1}}
\newcommand{\Charles}[1]{\marg{(Charles) #1}}

\ifHideFoot
\renewcommand{\Yano}[1]{}
\renewcommand{\Jeff}[1]{}
\renewcommand{\Charles}[1]{}
\fi

\DeclareMathOperator{\coniveau}{N}

\global\let\ker\undefined
\DeclareMathOperator{\ker}{ker}

\newcommand{\til}[1]{{\widetilde{#1}}}

\def\mmu{{\pmb\mu}}
\def\one{{\mathbf 1}}

\def\cross{\times}
\def\inv{^{-1}}

\def\cx{{\mathbb C}}

\def\gp{{\mathbb G}}
\def\rat{{\mathbb Q}}
\def\integ{{\mathbb Z}}

\def\iso{\cong}

\renewcommand{\bar}[1]{{\overline{#1}}}

\DeclareMathOperator{\alb}{Alb}

\DeclareMathOperator{\gal}{Gal}

\DeclareMathOperator{\pic}{Pic}

\DeclareMathOperator{\A}{A}
\DeclareMathOperator{\aut}{Aut}

\DeclareMathOperator{\chow}{CH}

\def\ra{\rightarrow}
\def\tensor{\otimes}
\newcommand{\st}[1]{\left\{#1\right\}}

\newcommand{\LKtrace}[1]{\underline{\underline{#1}}}

\newenvironment{alphabetize}{\begin{enumerate}
  
 }{\end{enumerate}}

\begin{document}
 

\title{Distinguished models of intermediate Jacobians}

 \author{Jeffrey D. Achter}
 \address{Colorado State University, Department of Mathematics,
  Fort Collins, CO 80523,
  USA}
 \email{j.achter@colostate.edu}
 
 \author{Sebastian Casalaina-Martin }
 \address{University of Colorado, Department of Mathematics, 
  Boulder, CO 80309, USA }
 \email{casa@math.colorado.edu}
 
 \author{Charles Vial}
 \address{Universit\"at Bielefeld, Fakult\"at f\"ur Mathematik, Postfach 100131, D-33501, Germany}
 \email{vial@math.uni-bielefeld.de}

 \thanks{The first author was partially supported by  grants from the
  NSA (H98230-14-1-0161, 
  H98230-15-1-0247 and H98230-16-1-0046).  The second author
  was partially supported by  a Simons Foundation
  Collaboration Grant for Mathematicians
  (317572) and NSA grant H98230-16-1-0053.  The third author was supported by 
  EPSRC Early Career Fellowship
  EP/K005545/1.}
 
 \date{\today}
 
 \maketitle
 
 %
 
 \vspace{-17pt}
 \begin{abstract}
We show that the image of the Abel--Jacobi
   map admits functorially a model over the field of definition, with the property
   that the Abel--Jacobi map is equivariant with respect to this
   model.   
    The cohomology of this abelian variety over the base field
   is isomorphic as a Galois representation to the deepest part of the
   coniveau filtration of the cohomology of the projective variety.
   Moreover, we show that this model over the base field is dominated
   by the Albanese variety of a product of components of the Hilbert
   scheme of the projective variety, and thus we answer a question of
   Mazur.  We also recover a result of Deligne on complete
   intersections of Hodge level~one.
 \end{abstract}
 
 
 \vspace{14pt}

  Let $X$ be a smooth projective variety defined over the complex numbers. Given a nonnegative integer $n$, denote $\operatorname{CH}^{n+1}(X)$ the Chow group of codimension-$(n+1)$ cycle classes on $X$, and denote $\operatorname{CH}^{n+1}(X)_{\mathrm{hom}}$ the kernel of the cycle class map $\operatorname{CH}^{n+1}(X) \to H^{2({n+1})}(X,\mathbb Z({n+1})).$ 
  In the seminal paper \cite{griffiths}, Griffiths defined a complex torus, the \emph{intermediate Jacobian},
   $J^{2n+1}(X)$ together with an \emph{Abel--Jacobi map} $$AJ : \operatorname{CH}^{n+1}(X)_{\mathrm{hom}} \to J^{2n+1}(X).$$
  While $J^{2n+1}(X)$ and the Abel--Jacobi map are transcendental in nature, the image of the Abel--Jacobi map restricted to   $\operatorname{A}^{{n+1}}(X)$, the sub-group of $\operatorname{CH}^{n+1}(X)$ consisting of algebraically trivial
  cycle classes, is a complex sub-torus $J^{2n+1}_a(X)$ of $J^{2n+1}(X)$ that is naturally endowed via the Hodge bilinear relations with a polarization, and hence is a complex abelian variety. The first cohomology group of $J^{2n+1}_a(X)$ is  naturally identified via the polarization with $\coniveau^nH^{2n+1}(X,\rat(n))$\,; i.e., the $n$-th Tate twist of the  $n$-th step in the
  geometric coniveau filtration (see \eqref{E:coniveau}). 
\medskip
  
  If now $X$ is a smooth projective variety defined over a sub-field $K \subseteq \cx$, it is natural to ask whether the complex abelian variety $J^{2n+1}_a(X_\cx)$ admits a model over $K$. In this paper, we prove  that $J^{2n+1}_a(X_\cx)$ admits a unique model over $K$ such that  $$AJ : \operatorname{A}^{n+1}(X_\cx) \to J^{2n+1}_a(X_\cx)$$ is $\operatorname{Aut}(\cx/K)$-equivariant, thereby generalizing the well-known cases of the Albanese map \linebreak $\operatorname{A}^{\dim X}(X_\cx) \to \operatorname{Alb}_{X_\cx}$ and of the Picard map $\operatorname{A}^{1}(X_\cx) \to \operatorname{Pic}_{X_{\mathbb C}}^0$, as well as the case of $AJ : \operatorname{A}^{2}(X_\cx) \to J^{3}_a(X_\cx)$ which was treated in our previous work \cite{ACMVdcg}.

  \begin{teoalpha}\label{T:main}
 Suppose $X$ is a smooth projective variety over a field $K\subseteq \cx$, and
 let $n$ be a nonnegative integer. Then $J^{2n+1}_a(X_\cx)$,  the complex abelian
 variety that is the image of the Abel--Jacobi map $AJ :
 \operatorname{A}^{n+1}(X_\cx) \to J^{2n+1}(X_\cx)$,  
 admits a distinguished model $J$ over $K$ such that the Abel--Jacobi map is 
 $\operatorname{Aut}(\cx/K)$-equivariant.
 Moreover, there is an algebraic
 correspondence $\Gamma \in \operatorname{CH}^{\dim (J)+n}(J\times_K
 X)$  inducing, for every prime number $\ell$,  a split inclusion of
 $\operatorname{Gal}(K)$-representations 
 \begin{equation}\label{E:Tmain-1}
  \Gamma_* : H^1(J_{\bar K},\rat_\ell) \hookrightarrow H^{2n+1}(X_{\bar
  	K},\rat_\ell(n))
 \end{equation}
 with image $\coniveau^nH^{2n+1}(X_{\bar K},\rat_\ell(n))$. 
  \end{teoalpha}

By Chow's rigidity theorem (see \cite[Thm.~3.19]{conradtrace}), an abelian variety $A/\mathbb C$ descends to  at most one model  defined over $\bar K$.  
On the other hand, an abelian variety $A/\bar K$ may descend to  more than one model   defined over $K$.  Nevertheless,  since $AJ:\operatorname{A}^{n+1}(X_{\mathbb C})\to
J_a^{2n+1}(X_{\mathbb C})$ is surjective, 
the abelian variety
$J_a^{2n+1}(X_{\mathbb C})$ admits at most one  structure of a scheme over $K$
such that $AJ$ is 
$\operatorname{Aut}(\cx/K)$-equivariant.  This is
the sense in which  $J_a^{2n+1}(X_{\mathbb C})$ admits a \emph{distinguished model}
over~$K$.  
\medskip

Our proof of Theorem \ref{T:main} uses a different strategy than we took in  \cite{ACMVdcg}, and as a result improves on the results of that paper in three ways\,:

1.~In \cite[Thm.~B]{ACMVdcg}, only the case $n=1$ of Theorem \ref{T:main} was treated.
An essential step in the proof in \cite[Thm.~B]{ACMVdcg} was a result of Murre \cite[Thm.~C]{murre83}, relying on the theorem of Merkurjev and Suslin, asserting that $J^3_a(X_\cx)$ is an \emph{algebraic representative}, meaning that it is universal among regular homomorphisms from $\operatorname{A}^2(X_\cx)$ (as defined in \S \ref{S:Tors-Gen}). 
In general, little is known about when higher intermediate Jacobians
are algebraic representatives, or even when algebraic representatives
exist. In this paper we  completely avoid the use of Murre's result,
or indeed the existence of an algebraic representative. Instead, we
use a new approach to show that for each $n$ there is a  model of
$J^{2n+1}_a(X_\cx)$ over $K$ which  makes the Abel--Jacobi map
$\operatorname{Aut}(\mathbb C/K)$-equivariant.  

2.~The results of \cite[Thm.~A]{ACMVdcg} concerning descent for $J^{2n+1}_a(X_\cx)$ for $n>1$ only show that the \emph{isogeny class} of $J^{2n+1}_a(X_\cx)$ descends to $K$, and this is  under the further restrictive assumption that the Abel--Jacobi map be surjective (or under some other constraint on the motive of $X$\,; see \cite[Thm.~2.1]{ACMVdcg}). In contrast, the present Theorem \ref{T:main} 
provides a \emph{distinguished model} of $J^{2n+1}_a(X_\cx)$ over $K$,
without any additional hypothesis.  Moreover, we show the assignment in Theorem
\ref{T:main} is functorial (Proposition \ref{P:functoriality}).  The
new technical input begins with Proposition \ref{P:niveauK}, which
shows that $J^{2n+1}_a(X_\cx)$ is dominated, via the induced action of
a correspondence defined over $K$, by the Jacobian of a pointed,
geometrically integral, smooth projective curve $C$ defined over
$K$, strengthening \cite[Prop.~1.3]{ACMVdcg}.  The key point is that  this strengthening,  together with the fact that
Bloch's map \cite{bloch79} factors through the Abel--Jacobi map on
torsion, makes it possible to show directly that $J^{2n+1}_a(X_\cx)$ admits a unique model over $K$ making the
Abel--Jacobi map $AJ : \operatorname{A}^{n+1}(X_\cx) \to
J^{2n+1}_a(X_\cx)$ Galois equivariant on torsion.  In short, avoiding the use of algebraic representatives, and the motivic techniques  employed in \cite{ACMVdcg}, we obtain a stronger result.
  We then make a
careful analysis of Galois equivariance for regular homomorphisms,
strengthening some statements in \cite{ACMVdcg}, to conclude that the
Abel--Jacobi map is Galois equivariant on all points -- and not merely on torsion points (Proposition
\ref{P:Tors-Gen})\,; this is crucial to the proof of Theorem \ref{Ta:MazQ1} below.
 
 3.~Finally, while a splitting in \cite[Thm.~A]{ACMVdcg} analogous to  \eqref{E:Tmain-1}  was established by some explicit computations involving correspondences, here we utilize  Andr\'e's powerful theory of \emph{motivated cycles} \cite{AndreIHES} in order to establish the more general splitting \eqref{E:Tmain-1}.   This also  provides a proof that the coniveau filtration splits (Corollary \ref{C:CNFSplit}), as well as a short motivic  proof that  the isogeny class of $J^{2n+1}_a(X_\cx)$ descends, without any of the restrictive hypotheses in \cite{ACMVdcg}.  
  \medskip

The structure of the proof of Theorem \ref{T:main} is broken into three parts.  First we give a proof of
Theorem \ref{T:main}, up to the statement of the splitting of the
inclusion, and where we focus only on the
$\operatorname{Aut}(\cx/K)$-equivariance of the Abel--Jacobi map on
torsion (Theorem \ref{T:JacDesc}). The proof of Theorem~\ref{T:JacDesc} relies on showing that
$\coniveau^nH^{2n+1}(X_{\bar K},\rat_\ell(n))$ is spanned via the
action of a correspondence over $K$ by the first cohomology group of
a pointed, geometrically integral curve\,; this is proved in
Proposition \ref{P:niveauK}.  Next, in \S\ref{S:Tors-Gen}, we show
that if the Abel--Jacobi map is
$\operatorname{Aut}(\mathbb C/K)$-equivariant on torsion, then it is
fully equivariant.  This is a consequence of more general results we
establish for surjective regular homomorphisms.  Finally, the
splitting of \eqref{E:Tmain-1} is then proved in Theorem \ref{T:mot}.
 Note 
that when $n=1$ the result of \cite[Thm.~A]{ACMVdcg} is more precise in
that the splitting of \eqref{E:Tmain-1} is shown to be induced by an
algebraic correspondence over $K$.
\medskip

  As a first application of Theorem \ref{T:main}, we recover a result of Deligne \cite{deligneniveau} regarding
  intermediate Jacobians of complete intersections of Hodge level $1$ (\S
  \ref{S:Deligne}).  

\medskip

Another application  is to the following question due to Barry Mazur.
  Given an effective  polarizable weight-$1$ $\mathbb Q$-Hodge structure $V$, there is a complex 
  abelian variety  $A$ (determined up to isogeny) so that  $H^1(A,\mathbb
  Q)$ gives a weight-$1$ $\mathbb Q$-Hodge structure isomorphic to $V$.  
  On the other hand, 
  let $K$ be a field, and let $\ell$ be a prime number (not equal to the
  characteristic of the field).
  It is not known (even for $K=\mathbb Q$) whether  given  an effective
  polarizable weight-$1$
  $\operatorname{Gal}(K)$-representation $V_\ell$ over $ \mathbb Q_\ell$,  there
  is an abelian variety $A/K$ such that $H^1(A_{\bar K},\mathbb Q_\ell)$
  is isomorphic to $V_\ell$.    
  A \emph{phantom abelian variety for $V_\ell$} is an abelian variety $A/K$
  together with an isomorphism of $\operatorname{Gal}(K)$-representations 
  $$
  \xymatrix{
 H^1(A_{\bar K},\mathbb Q_\ell)\ar[r]^<>(0.5){\cong} & V_\ell.
  }
  $$
  Such an abelian variety, if it exists, is determined up to isogeny\,; this is
  called the \emph{phantom isogeny class for $V_\ell$}.  
  Mazur asks the following question 
\cite[p.38]{mazurprobICCM} \,:
\noindent\emph{Let  $X$ be a smooth projective variety over a field $K\subseteq \cx$, and let
 $n$ be a nonnegative integer.  If $H^{2n+1}(X_{\mathbb C},\mathbb Q)$ has Hodge
 coniveau $n$ (i.e., $H^{2n+1}(X_{\mathbb C},\mathbb
 C)=H^{n,n+1}(X) \oplus H^{n+1,n}(X)$), does there exist a phantom abelian
 variety for $H^{2n+1}(X_{\bar K},\mathbb Q_\ell(n))$?}

 Theorem \ref{T:main}
  answers this affirmatively  under the stronger, but according to the generalized
  Hodge conjecture equivalent, assumption that the Abel--Jacobi map  $AJ :
  \operatorname{A}^{n+1}(X_\cx) \to J^{2n+1}(X_\cx)$ is surjective. This assumption
  is known to hold in many cases (e.g.,
  uniruled threefolds).    Theorem~\ref{T:main} in fact shows a  stronger statement, namely that the \emph{distinguished}
  model over $K$ of  the image of the Abel--Jacobi map  $AJ :
  \operatorname{A}^{n+1}(X_\cx) \to J^{2n+1}(X_\cx)$ provides a phantom abelian variety for the 
  $\operatorname{Gal}(K)$-representation $\coniveau^nH^{2n+1}(X_{\bar
 K},\rat_\ell(n))$.  
  Moreover, the
  arguments via motivated cycles of Section \ref{S:Andre} give a second proof of the  existence of a
  phantom abelian variety, although not the descent of the image of the
  Abel--Jacobi map.    In summary, these results  strengthen our answer to Mazur's question, given in \cite{ACMVdcg}.

  \vskip .2 cm 
  
  As another application, we  provide an answer to a second  question of Mazur, which was not addressed in \cite{ACMVdcg}.  Over the complex numbers the image of the Abel--Jacobi map is dominated by
  Albaneses of resolutions of singularities of products of irreducible components
  of Hilbert  
  schemes.  
  Since Hilbert schemes are functorial, and in particular  defined over $K$, and
  since the image of the Abel--Jacobi map descends to $K$, one might expect the
  phantom abelian variety to be linked to the  Albanese of a Hilbert  scheme.
  Motivated by concrete examples where this holds (e.g., the intermediate Jacobian
  of a smooth cubic threefold $X$ is the Albanese variety of the Fano variety of
  lines on $X$ \cite{CG}),
  Mazur asks the following
  question
\cite[Que.~1]{mazurprobICCM}\,:
\emph{Can this phantom abelian variety  be constructed as -- or at least in terms of
 -- the Albanese variety of some Hilbert scheme geometrically attached to $X$?}
  We provide an affirmative answer    for 
  $\coniveau^nH^{2n+1}(X_{\bar K},\rat_\ell(n))$\,:
  
  \begin{teoalpha}\label{Ta:MazQ1}
 Suppose $X$ is a smooth projective variety over a field $K\subseteq \cx$.  Then
 the phantom abelian variety $J/K$ for $\coniveau^nH^{2n+1}(X_{\bar
 	K},\rat_\ell(n))$ given in Theorem \ref{T:main} is dominated by the Albanese
 variety of  (a finite product of resolutions of singularities   of some  finite
 number of)
 components of a Hilbert scheme parameterizing codimension-$(n+1)$ subschemes of 
 $X$ over
 $K$. 
  \end{teoalpha}

  The proof of the theorem, given in \S \ref{S:Hilb} (Theorem \ref{T:MazQ1}), uses in an essential way the 
  $\operatorname{Aut}(\cx/K)$-equivariance of the Abel--Jacobi map as stated in Theorem \ref{T:main}. For the sake of generality, the proof  is
  framed in the language  of Galois-equivariant regular homomorphisms, as
  described  in \cite[\S 4]{ACMVdcg}.    As a consequence, some related results
  are obtained for algebraic representatives of smooth projective varieties over
  perfect fields of arbitrary characteristic.
\medskip

  For concreteness, we mention the following consequence of  Theorems \ref{T:main} and 
  \ref{Ta:MazQ1}, providing a complete answer to Mazur's questions for uniruled threefolds (see \S \ref{S:Hilb})\,:

  \begin{coralpha}\label{Ca:MazQ1uni}
 Suppose $X$ is a smooth projective threefold over a field   $K\subseteq \mathbb
 C$ and assume that $X_\cx$ is uniruled.
 Then the intermediate Jacobian $J^3(X_{\mathbb C})$ descends to an abelian
 variety $J/K$,   which is  a  phantom abelian variety  for $H^{3}(X_{\bar
 	K},\rat_\ell(1))$, and is dominated by the Albanese variety of  (a product of
 resolutions of singularities  of a finite number of) components of a Hilbert
 scheme
 parameterizing  dimension-$1$ subschemes of  $X$ over $K$.
  \end{coralpha}

  \vskip .2cm 
  
  \noindent \textbf{Acknowledgments.} 
  We would like to thank Ofer Gabber for comments that were instrumental in
  arriving at Theorem \ref{T:main}.  We also thank the referee for helpful suggestions.

  \vskip .2 cm

  \noindent \textbf{Conventions.}  We use the same conventions as in
  \cite{ACMVdcg}.  
  A \emph{variety} over a field is a geometrically reduced separated
  scheme of finite type over that field.
  A \emph{curve}
  (resp.~\emph{surface})  is a variety of pure dimension $1$ (resp.~$2$).
  Given a variety $X$,   $\operatorname{CH}^i(X)$  denotes  the Chow
  group of codimension~$i$ cycles modulo rational equivalence, and
  $\operatorname{A}^i(X)\subseteq \chow^i(X)$  denotes the subgroup of cycles
  algebraically equivalent to $0$.  If $X$ is a smooth projective variety over the
  complex numbers, then we denote by $J^{2n+1}(X) =
  \operatorname{F}^{n+1}H^{2n+1}(X,\cx)\backslash H^{2n+1}(X,\cx)/H^{2n+1}(X,\integ)$ 
     the complex torus that is the
  $(2n+1)$-th intermediate Jacobian of $X$, and  we denote $J_a^{2n+1}(X)$  the
  image of the Abel--Jacobi map $\operatorname{A}^{n+1}(X) \to
  J^{2n+1}(X)$. A
  choice of polarization on $X$ naturally endows  the complex torus
  $J_a^{2n+1}(X)$ with the structure of a polarized complex abelian
  variety, and $H^1(J_a^{2n+1}(X),\rat) \iso
  \coniveau^nH^{2n+1}(X,\rat)(n)$.  If $C/K$ is a smooth projective geometrically irreducible curve over a field, we will sometimes write $J(C)$ for $\pic^\circ_{C/K}$.
Given a field $K$, we denote by $\bar K$ a separable closure.
  Finally, given an abelian group $A$, we denote by $A[N]$ the kernel of the
  multiplication-by-$N$ map\,; and if $A$ is an abelian group scheme over a field
  $K$, we write $A[N]$ for $A(\bar K)[N]$.

 \section{A result on cohomology}
 
 The main point of this section is to prove  Proposition \ref{P:niveauK}, 
 strengthening \cite[Prop.~1.3]{ACMVdcg}.  Recall that if $X$ is a smooth
 projective variety over a field $K$, then the geometric coniveau filtration 
 $\coniveau^\nu H^{i}(X_{\bar
  K},\rat_\ell)$ is defined by\,:
 
 \begin{equation} \label{E:coniveau}
 \coniveau^\nu H^{i}(X_{\bar
  K},\rat_\ell) := \sum_{\substack{ Z\subseteq
   X\\\text{closed, codim }\ge \nu}} \operatorname{ker}\left(
 H^{i}(X_{\bar{K}},\rat_\ell) \rightarrow
 H^{i}(X_{\bar{K}}
 \backslash Z_{\bar{K}}, \rat_\ell)\right).
 \end{equation}
 If $K=\mathbb C$, the geometric coniveau filtration $\coniveau^\nu
 H^{i}(X,\rat)$ is defined similarly.  We direct the reader to \cite[\S
 1.2]{ACMVdcg} for a review of some of the properties we use here.   Sometimes, we
 will abuse notation slightly and denote the $r$-th Tate twist of step $\nu$ in
 the geometric coniveau filtration by $\coniveau^\nu H^{i}(X_{\bar
  K},\rat_\ell(r)):=(\coniveau^\nu H^{i}(X_{\bar
  K},\rat_\ell))\otimes_{\mathbb Q_\ell} \mathbb Q_\ell(r)$, and similarly for
 Betti cohomology.

 \begin{pro} \label{P:niveauK}
  Suppose $X$ is a smooth  projective variety  
  over a field $K\subseteq \cx$,
  and let $n$ be a nonnegative integer. Then there
  exist a geometrically integral smooth
  projective curve
  $C$ over $K$, admitting a $K$-point,  and a correspondence $\gamma \in
  \operatorname{CH}^{n+1}(C\times_K
  X)_{\mathbb{Q}}$  such that for all  primes $\ell$,  the
  induced morphism of
  $\operatorname{Gal}(K)$-representations
  $$
  \gamma_*: H^1(C_{\overline K},{\mathbb{Q}}_\ell) \rightarrow
  H^{2n+1}(X_{\overline K},{\mathbb{Q}}_\ell(n))
  $$ 
  has image
  $\coniveau^nH^{2n+1}(X_{\overline K},{\mathbb{Q}}_\ell(n))$. 
  Likewise, the morphism of Hodge structures
  $$
  \gamma_*: H^1(C_{\mathbb{C}},{\mathbb{Q}}) \rightarrow
  H^{2n+1}(X_{\mathbb{C}},{\mathbb{Q}}(n))
  $$ 
  has image
  $\coniveau^nH^{2n+1}(X_{\mathbb{C}},{\mathbb{Q}}(n))$\,; in particular, the
  image of $\gamma_*
  : J(C_\cx) \to J^{2n+1}(X_\cx)$ is $J^{2n+1}_a(X_\cx)$.
 \end{pro}
 
 \begin{rem}\label{R:P:niveau}
  The result \cite[Prop.~1.3]{ACMVdcg} differs from Proposition \ref{P:niveauK}
  only in the sense that is it not
  shown there  that $C$ can be taken to admit a $K$-rational point  or to be
  geometrically integral.
 \end{rem}

 There are three main  ingredients in the proof
 of Proposition \ref{P:niveauK}\,: the Bertini theorems,  the  Lefschetz type result in
 Lemma \ref{L:ladicLef} below describing cohomology in degree $1$, and
 Proposition \ref{P:curveCoh} regarding the cohomology of curves.   
 While we expect Proposition \ref{P:curveCoh} is well-known, 
 for lack of a reference  we include a proof in  Appendix \ref{A:Appendix}.

 \begin{lem}[Lefschetz] \label{L:ladicLef}
  Suppose $X$ is a smooth  projective variety  
  over a field $K$ with separable closure $\overline K$. 
   There exist a  smooth curve 
  $C\hookrightarrow X$ over $K$, which is a (general) linear section 
  for an appropriate projective embedding of $X$, 
  and a correspondence $\gamma \in
  \operatorname{CH}^{1}(C\times_K
  X)_{\mathbb{Q}}$  such that
  for all $\ell\not = \operatorname{char}(K)$,
  the induced morphism of
  $\operatorname{Gal}(K)$-representations
  $$
\xymatrix{
  \gamma_*: H^1(C_{\overline K},{\mathbb{Q}}_\ell) \ar@{->>}[r]&
  H^{1}(X_{\overline K},{\mathbb{Q}}_\ell)
}
  $$ 
  is surjective.  
  Moreover, if $X$ is geometrically integral (resp.~admits a $K$-point), then $C$
  can be taken to be geometrically integral (resp.~to admit a $K$-point).  
 \end{lem}

 \begin{proof}
  By Bertini \cite{poonen}, let $\iota : C \hookrightarrow X$ be a
  one-dimensional smooth general linear section of an appropriate projective
  embedding of $X$. 
  Note that by the irreducible Bertini theorems \cite{charlespoonen},  if $X$ is
  geometrically integral (resp.~admits a $K$-point), then $C$ can also be taken to
  be
  geometrically integral (resp.~to admit a $K$-point)\,; (see e.g., \cite[Thm. B.1]{ACMV} for the version we use here).

  The hard Lefschetz 
  theorem \cite[Thm.\ 4.1.1]{deligneweil2} states that intersecting with $C$
  yields an isomorphism 
  $$
  \iota_*\iota^*:H^1(X_{\bar K},\mathbb Q_\ell) \hookrightarrow H^1(C_{\bar
   K},\mathbb
  Q_\ell)  \twoheadrightarrow H^{2\dim X-1}(X_{\bar
   K},\mathbb Q_\ell(\dim X-1)).
  $$
  The Lefschetz Standard Conjecture is known for $\ell$-adic cohomology and for
  Betti cohomology in degree $\leq 1$ (see \cite[Thm.~2A9(5)]{kleiman}),
  meaning in our case that the map
  $\left(\iota_*\iota^*\right)^{-1}$ 
  is induced by a correspondence, say  $\Lambda \in \operatorname{CH}^1(X\times_K
  X)_\rat$.   
  Therefore, the composition 
  \begin{equation*}\label{E:LCoNi}
  \xymatrix@C=2em{
   H^1(C_{\bar K},\mathbb Q_\ell)
   \ar@{->>}[r]_<>(0.5){\iota_*}^<>(0.5){(\Gamma_{\iota})_*} &  
   H^{2\dim X-1}(X_{\bar K},\mathbb Q_\ell(\dim X-1)) \ar@{->}[r]^<>(0.5){
    \Lambda_*}_<>(0.5)\cong & H^1(X_{\bar K},\mathbb Q_\ell)
  }
  \end{equation*}
is surjective and is induced by the correspondence $\gamma := \Lambda \circ
  \Gamma_\iota$, where $\Gamma_\iota$ denotes the graph of $\iota$ . 
 \end{proof}

 \begin{proof}[Proof of Proposition \ref{P:niveauK}]
  Up to working component-wise, we can and do assume
  that $X$ is irreducible, say of dimension $d_X$.
  Since 
  $K\subseteq \mathbb C$, 
  we have from the characterization of coniveau (see e.g.,  \cite[(1.2)]{ACMVdcg})
  that there exist  a 
  smooth projective variety $Y$ (possibly disconnected)  of pure dimension
  $d_Y=d_X-n$ over $K$,
  and a
  $K$-morphism $f:Y\to X$ such that 
  $$ \operatorname{N}^nH^{2n+1}(X_{\bar K},{\mathbb{Q}}_\ell(n)) =
  \mathrm{Im}\left(
  f_*:H^1({Y}_{\bar K},{\mathbb{Q}}_\ell)
  \rightarrow
  H^{2n+1}(X_{\bar K},{\mathbb{Q}}_\ell(n))\right).
  $$ 
  Using Lemma \ref{L:ladicLef} applied to $Y$, there exist a smooth projective
  curve
  $C$ over $K$ (possibly disconnected) and a correspondence $\Gamma\in
  \operatorname{CH}^1(C\times_KY)_{\mathbb Q}$ such that the composition 
  \begin{equation*}
  \xymatrix@C=2em{
   H^1(C_{\bar K},\mathbb Q_\ell) \ar@{->>}[r]^<>(0.5){\Gamma_*} & H^1(Y_{\bar
    K},\mathbb Q_\ell)
   \ar[r]^<>(0.5){f_*} & H^{2n+1}(X_{\bar K},\mathbb Q_\ell(n))\\
  }
  \end{equation*} 
  has image
  $\operatorname{N}^nH^{2n+1}(X_{\bar K},{\mathbb{Q}}_\ell(n))$. 
  
  As recalled in Proposition \ref{P:curveCoh},
  there is a morphism $\beta:C\to \pic^\circ_{C/K}$ inducing an  
  isomorphism 
  $\beta^*=(\Gamma_\beta^t)_*:H^1(\pic^\circ_{C_{\bar K}/\bar
   K},\mathbb Q_\ell)\to 
  H^1(C_{\bar K},\mathbb Q_\ell)$. 
  Observe that $\pic^\circ_{C/K}$ is
  geometrically integral and admits a $K$-point.  Lemma
  \ref{L:ladicLef} yields a smooth geometrically integral curve $D/K$ endowed
  with a $K$-point, and a
  surjection  $
  H^1(D_{\bar K},\mathbb Q_\ell)\twoheadrightarrow
  H^1(\pic^\circ_{C_{\bar K}/\bar K},\mathbb Q_\ell)
  $
  induced by a
  correspondence $\widetilde  \Gamma$ over $K$.   The composition
  \begin{equation*}
  \xymatrix@C=1.4em{
   H^1(D_{\bar K},\mathbb Q_\ell) \ar@{->>}[r]^<>(0.5){\widetilde 
    \Gamma_*}_<>(0.5){\operatorname{Lef}} & H^1(\pic^\circ_{C_{\bar
     K}/\bar K},\mathbb Q_\ell)
   \ar[r]^<>(0.5){(\Gamma_\beta^t)_*}_<>(0.5){\cong}&H^1(C_{\bar K},\mathbb Q_\ell)
   \ar@{->>}[r]^<>(0.5){\Gamma_*}_<>(0.5){\operatorname{Lef}}  & H^1(Y_{\bar
    K},\mathbb Q_\ell)
   \ar@{->>}[r]^<>(0.5){f_*}_<>(0.5){\operatorname{Def}} &
   \coniveau^nH^{2n+1}(X_{\bar K},\mathbb
   Q_\ell(n)),
  }
  \end{equation*}
  induced by the  associated composition of correspondences $\gamma$, provides the
  desired surjection 
$$\xymatrix{\gamma_*:H^1(D_{\bar K},\mathbb Q_\ell) \ar@{->>}[r]&
  \coniveau^n H^{2n+1}(X_{\bar K},\mathbb Q_\ell(n))}.$$

Finally, the compatibility of the comparison isomorphisms in
cohomology with Gysin maps and the action of correspondences
(see  e.g., \cite[\S 1.2]{ACMVdcg}), or simply rehashing the argument above after pull-back to $\mathbb C$,   establishes
  that  the image  of the induced morphism of Hodge structures 
  $
  \gamma_*:
  H^1(D_{\mathbb C},{\mathbb{Q}})
  \rightarrow
  H^{2n+1}(X_{\mathbb C},{\mathbb{Q}}(n))
  $
  is
  $\operatorname{N}^nH^{2n+1}(X_{\mathbb C},{\mathbb{Q}}(n)) =
  H^1(J^{2n+1}_a(X_{\mathbb C},\rat))$.  Using the equivalence of
  categories between polarizable effective weight one Hodge structures
  and complex abelian varieties, we see that this morphism of Hodge
  structures is induced by a surjection of abelian varieties $\gamma_*:
  J(D_\cx) \twoheadrightarrow J^{2n+1}_a(X_\cx)$.
 \end{proof}

 \section{Proof of Theorem \ref{T:main}\,: Part I, descent of the
   image of the Abel--Jacobi map}

 In this section we establish the following theorem, proving the first part of
 Theorem \ref{T:main}.

 \begin{teo}\label{T:JacDesc}
  Suppose $X$ is a smooth projective variety over a field $K\subseteq \mathbb C$,
  and let
  $n$ be a nonnegative integer. Then the image  of the
  Abel--Jacobi map $J^{2n+1}_a(X_\mathbb C)$
  admits a distinguished model $J$ over $K$ such that the induced map
  $AJ[N]:\operatorname{A}^{n+1}(X_{\mathbb C})[N]\to
  J_a^{2n+1}(X_{\mathbb C})[N]$ on $N$-torsion  is $\operatorname{Aut}(\cx/K)$-equivariant for all
  positive integers $N$.
  
  Moreover, there is a correspondence $\Gamma\in \operatorname{CH}^{\dim
   (J)+n}(J\times_K X)$ such that for each prime number $\ell$,  we have that   $\Gamma$ induces an  inclusion of $\operatorname{Gal}(K)$-representations  
  \begin{equation}\label{E:cycles4}
  \xymatrix{ 
   H^1(J_{\bar K},\mathbb Q_\ell) \ar@{^(->}[r]^<>(0.5){\Gamma_*}&
   H^{2n+1}(X_{\bar K},\mathbb Q_\ell(n)),
  }
  \end{equation}
  with image $\operatorname{N} ^n H^{2n+1}(X_{\bar K},\mathbb Q_\ell(n))$. 
 \end{teo}

 We will prove the theorem in several steps contained in the following
 subsections.
 
 \begin{rem}\label{R:Distinguished} 
 	As explained below the statement of Theorem \ref{T:main}, an abelian variety over $\cx$ may admit several models over $K$, if it admits any. However, it admits at most one model over $K$ such that the induced map
 	$AJ[N] :\operatorname{A}^{n+1}(X_{\mathbb C})[N]\to
 	J_a^{2n+1}(X_{\mathbb C})[N]$ on $N$-torsion  is $\operatorname{Aut}(\cx/K)$-equivariant for all
 	positive integers $N$.
  Indeed, by Chow's rigidity theorem (see \cite[Thm.~3.19]{conradtrace}), an abelian variety $A/\mathbb C$ descends to  at most one model  defined over $\bar K$\,; moreover, there is at most one model of $A$ defined over $K$ that induces a  given action of $\operatorname{Gal}(K)$ on the $\bar K$-points of $A$. 
 Therefore, since $AJ[N]:\operatorname{A}^{n+1}(X_{\mathbb C})[N]\to
  J_a^{2n+1}(X_{\mathbb C})[N]$ is surjective for all $N$ not divisible by a finite number of fixed primes (this is a general fact about regular homomorphisms\,; see \S \ref{S:Tors-Gen} and Lemma \ref{L:Tor-to-Tor}(b)), and since torsion points of order  not divisible by a finite number of fixed primes are dense,
   the abelian variety
  $J_a^{2n+1}(X_{\mathbb C})$ admits at most one  structure of a scheme over $K$
  such that $AJ[N]$ is 
  $\operatorname{Aut}(\cx/K)$-equivariant for all positive integers $N$ not divisible by the finite number of given primes.  This is
  the sense in which  $J_a^{2n+1}(X_{\mathbb C})$ admits a distinguished model
  over $K$.  
 \end{rem}

 \subsection{Chow rigidity and $L/K$-trace\,: descending from $\mathbb C$ to
  $\bar
  K$} \label{S:rigidity}
 The first step in the proof consists in using Chow rigidity and $\mathbb C/\bar K$-trace
 to descend the image of the Abel--Jacobi map from $\mathbb C$ to $\bar K$.   We
 follow the treatment in \cite{conradtrace}, and refer the reader to \cite[\S
 3.3]{ACMVdcg} where we review the theory in the setting we use here.  
 
 For the convenience of the reader, we briefly recall a few points.  We
 focus on the case where $L/K$ is an extension of algebraically closed fields of
 characteristic $0$. 
 First, we reiterate that by  Chow's rigidity theorem (see \cite[Thm.~3.19]{conradtrace}), an abelian variety $B/L$ descends to  at most one model, up to isomorphism, defined over $K$.  
   Given an abelian variety $B$ defined over $L$, while $B$ need not descend to $K$, there is  
 \cite[Thm.~6.2, Thm.~6.4, Thm.~6.12, p.72, p.76, Thm.~3.19]{conradtrace} 
  an abelian variety $\LKtrace B$ defined over $K$ equipped with an injective 
 homomorphism of abelian varieties 
 $$\xymatrix{
  \LKtrace B_L \ar@{^(->}[r]^\tau&  B
 }
 $$ (together called the $L/K$-trace) with the property that for any abelian
 variety $A/K$, base change gives an identification
\begin{align*}
 \operatorname{Hom}_{\mathsf {Ab}/K}(A,\LKtrace B) &= \operatorname{Hom}_{\mathsf
  {Ab}/L}(A_L,B) \\
f &\mapsto \tau \circ f_L.
\end{align*}
It follows
 that if there is an abelian variety $A/K$ and a surjective homomorphism $A_L\to
 B$, then $\tau$ is surjective and hence an isomorphism\,; in other words, $B$
 descends to $K$ (as $\LKtrace{B}$).

 \begin{proof}[Proof of Theorem \ref{T:JacDesc}, Step 1\,: $J^{2n+1}_a(X_{\mathbb
   C})$ descends to $\bar K$.]
  In the notation of Theorem \ref{T:JacDesc}, we wish to show that
  $J^{2n+1}_a(X_{\mathbb
   C})$ descends to an abelian variety over $\bar K$.   We have shown in
  Proposition
  \ref{P:niveauK} that there exist a smooth projective geometrically integral
  curve
  $C/K$, admitting a $K$-point, and a correspondence  $\gamma \in 
  \operatorname{CH}^{n+1}(C\times_KX)_{\mathbb{Q}}$ 
  which induces a surjection
  $\gamma_*:J(C_{\mathbb C}) \twoheadrightarrow  J^{2n+1}_a(X_{\mathbb C})$.  
  Thus from the theory of the ($\mathbb C/\bar{K}$)-trace, 
  and the fact that  $J(C_{\mathbb C})= J(C_{\bar{K}})_{\mathbb C}$ is defined
  over $\bar K$,  $J^{2n+1}_a(X_{\mathbb C})$ descends to $\bar K$ as
  its ($\mathbb C/\bar{K}$)-trace $ \LKtrace J^{2n+1}_a(X_{\mathbb C})$, and 
  there is a surjective homomorphism of abelian varieties over $\bar K$
  $$
  \xymatrix{
   J(C_{\bar{K}}) \ar@{->>}[r]^<>(0.5){\LKtrace {\gamma_*}}&   \LKtrace
   J^{2n+1}_a(X_{\mathbb C}).
  }
  $$
  Moreover, the Abel--Jacobi map on torsion $AJ[N] :
  \operatorname{A}^{n+1}(X_{\cx})[N] \to J_a^{2n+1}(X_\cx)[N]$ is
  $\operatorname{Aut}(\cx/\bar K)$-equivariant for all positive integers $N$.
  Indeed,  $\operatorname{Aut}(\cx/\bar K)$ acts trivially on
  $J_a^{2n+1}(X_\cx)[N]=\LKtrace J^{2n+1}_a(X_{\mathbb C})[N]$ 
  and it also acts trivially on $\operatorname{A}^{n+1}(X_{\cx})[N]$ by Lecomte's
  rigidity theorem \cite{lecomte86} (see e.g.,  \cite[Thm.~3.8(b)]{ACMVdcg}).
 \end{proof}

 \subsection{Descending from $\bar K$ to $K$}\label{S:algdescent}
 
 In the notation of Theorem \ref{T:JacDesc}, we have found a smooth projective
 geometrically integral curve $C/K$, admitting a $K$-point, and a correspondence
 $\gamma \in \operatorname{CH}^{n+1}(C\times_KX)_{\mathbb{Q}}$ inducing 
 a surjective homomorphism of abelian varieties over $\bar K$
 \begin{equation}
 \label{E:Pdef}
 \xymatrix{
  0\ar[r]& P \ar[r]& J(C_{\bar{K}}) \ar@{->>}[r]^<>(0.5){\LKtrace {\gamma_*}}&  
  \LKtrace J^{2n+1}_a(X_{\mathbb C}) \ar[r] & 0
 }
 \end{equation}
 
 \noindent where $P$ is defined to be the kernel. 
 We will show that 
 $ \LKtrace J^{2n+1}_a(X_{\mathbb C})$ descends to an abelian variety $J$ over $K$ by showing that $P$
 descends to $K$, using the following
 elementary criterion\,:

 \begin{lem} \label{L:torsdescends}
  Let $A/K$ be an abelian variety over a perfect field $K$,  let $\Omega/K$
  be an algebraically closed extension field, and let $\bar A = A_{\Omega}$.
  Suppose that $\bar B
  \subset \bar A$ is a closed sub-group scheme.  Then $\bar B_{{\rm
    red}}$ descends to a sub-group scheme over $K$ if and only if, for
  each natural number $N$, we have $\bar B[N](\Omega)$ is stable under
  $\aut(\Omega/K)$.\qed
 \end{lem}

\begin{proof}
It is well-known that, since the fixed field of $\Omega$ under $\aut(\Omega/K)$ is $K$
itself, a subvariety $W$ of $\bar A$ descends to $K$ if and only if
$W(\Omega)$ is stable under $\aut(\Omega/K)$ (e.g.,
\cite[Prop.~6.8]{milneAG}).  In fact, to show $W$ descends to $K$ it
suffices to verify that there is a 
Zariski-dense subset $S \subset W(\Omega)$ which is stable under
$\aut(\Omega/K)$. (Indeed, if $\sigma \in \aut(\Omega/K)$, then $W^\sigma$ contains the Zariski closure of $S^\sigma$, which is $W$ itself.)
Now use the fact that, over an algebraically closed
field, torsion points are Zariski dense in any abelian variety or
\'etale group scheme.
\end{proof}

 \begin{proof}[Proof of Theorem \ref{T:JacDesc}, Step 2\,:
  $\LKtrace{J}^{2n+1}_a(X_{\mathbb C})$ descends  to $K$.]
  We wish to show that the abelian variety $\LKtrace{J}^{2n+1}_a(X_{\mathbb C})$
  over $\bar  K$, obtained in Step 1 of the proof, descends to an abelian variety
  over $K$.  
  In the notation of  Step 1, let $P$ be
  the kernel of $\LKtrace{\gamma_*}$, as in \eqref{E:Pdef}.
  We use the criterion of Lemma \ref{L:torsdescends} to show that $P$
  descends to $K$.  
  To this end, let $N$ be a natural number.
  We have a commutative diagram of abelian
  groups\,:

\begin{equation}\label{E:DescBigD}
\vcenter{\xymatrix@R=.6em@C=1em{
&P[N] \ar@{^(->}[dd]\\
P_\cx[N] \ar@{=}[ur] \ar@{^(->}[dd] &\\
& J(C_{\bar K})[N] \ar[rr]^{\simeq} \ar@{->}'[d]^{\LKtrace \gamma_{*,N}}[dd]&&H^1_{\text{\'et}}(C_{\bar K},\mmu_N) \ar[dd]^{\gamma_{*,N}} \\
J(C_\cx)[N] \ar[dd]_{\gamma_{*,N}} \ar@{=}[ur]\ar[rr]_(0.6){\simeq} &&
H^1_{\text{an}}(C_\cx,\mmu_N)\ar@{=}[ur] \ar[dd]_{\gamma_{*,N}}\\
& \LKtrace{J_a}^{2n+1}(X_\cx)[N]&&H^{2n+1}_{\text{\'et}}(X_{\bar K},\mmu_N^{\otimes(n+1)}) \\
J_a^{2n+1}(X_\cx)[N] \ar@{=}[ur] \ar@{^(->}[r] & J^{2n+1}(X_\cx)[N]
\ar@{=}[r ] & H^{2n+1}_{\text{an}}(X_\cx,\mmu_N^{\otimes(n+1)})\ar@{=}[ur]
}}
\end{equation}

\noindent Here, the identification $J^{2n+1}(X_\cx)[N]= H^{2n+1}_{\text{an}}(X_\cx,\mmu_N^{\otimes(n+1)})$ is  given by the  definition of the intermediate Jacobian $J^{2n+1}(X_\cx)$, since $H^{2n+1}_{\text{an}}(X_\cx,\mmu_N^{\otimes(n+1)})=H^{2n+1}_{\text{an}}(X_\cx,\mathbb Z/n\mathbb Z)=H^{2n+1}_{\text{an}}(X_\cx,\frac{1}{n}\mathbb Z/\mathbb Z)$.
   The key point then is that, by commutativity,  the composition of 
  arrows along the back of the diagram
  \begin{equation}\label{E:back}
  \xymatrix{
   J(C_{\bar K})[N]   \ar[r]^<>(0.5)\simeq &  H_{\text{\'et}}^1(C_{\bar
    K},\mmu_{N})   \ar[r]^<>(0.5){\gamma_{*,N}} &
   H_{\text{\'et}}^{2n+1}(X_{\bar K},\mmu_{N}^{\otimes {(n+1)}}) 
  }
  \end{equation}
  has the same kernel as the  arrow
  $\LKtrace \gamma_{*,N}$, namely $P[N]$.  Moreover, 
  each arrow of the composition \eqref{E:back}
  is  $\operatorname{Gal}(K)$-equivariant.
  Therefore, $P[N] = \ker \LKtrace \gamma_{*,N}$ is
  $\gal(K)$-stable for each $N$, and $P$ descends to~$K$. Therefore  the abelian variety $\LKtrace{J}^{2n+1}_a(X_{\mathbb C})$
  over $\bar  K$ admits a model $J$ over $K$ such that the $\bar K$-homomorphism   $\LKtrace {\gamma_*} : J(C_{\bar{K}}) \to   \LKtrace
  	J^{2n+1}_a(X_{\mathbb C})$ descends to a $K$-homomorphism $f : J(C) \to J$.
 \end{proof}

 \subsection{The Abel--Jacobi map is Galois-equivariant on torsion}
 In the notation of Theorem \ref{T:JacDesc}, we have so far established that
 $J^{2n+1}_a(X_{\mathbb C})$ descends to an abelian variety $J$ over $K$.   We now
 wish to show that with respect to this given structure as a $K$-scheme, the
 Abel--Jacobi map on torsion 
\begin{equation}
\label{E:AJtorsion}
 \xymatrix{AJ :
 \operatorname{A}^{n+1}(X_{\cx})[N] \ar[r]& J^{2n+1}_a(X_\cx)[N] = J[N]}
\end{equation}
is
 $\operatorname{Aut}(\cx/K)$-equivariant.   In Step 1, we already showed
 that $AJ$ is $\operatorname{Aut}(\cx/\bar K)$-equivariant when restricted to
 torsion. Therefore, in order to conclude, it only remains to prove  that the map
 $\LKtrace{AJ} : \operatorname{A}^{n+1}(X_{\bar K})[N] \to
 J[N]$ is
 $\operatorname{Gal}(\bar K/K)$-equivariant.

 \vskip .2 cm 
 For future reference, we have the following elementary lemma. 
 
 \begin{lem}\label{L:ez-emma}
  Let $G$ be a group and let $A,B,C$ be  $G$-modules. Let
  $\phi : A \to B$ and $ \psi : B \to C$   be homomorphisms of abelian groups. We
  have\,:
  \begin{alphabetize}
   \item If $\phi$ is surjective and if $\phi$ and $\psi\circ \phi $ are
   $G$-equivariant,
   then $\psi$ is $G$-equivariant.
   \item If $\psi $ is injective and if $\psi$ and $\psi\circ \phi $ are
   $G$-equivariant, then
   $\phi$ is $G$-equivariant.  \qed
  \end{alphabetize} 
 \end{lem}

 \begin{proof}[Proof of Theorem \ref{T:JacDesc}, Step 3\,: The Abel--Jacobi map
  is equivariant on torsion] Fix  $J/K$ to be the   model of
  $J^{2n+1}_a(X_{\mathbb C})$ from Step 2.  We wish to show that for any positive
  integer $N$,  the restriction \eqref{E:AJtorsion} of the Abel--Jacobi map to $N$-torsion  is $\aut(\cx/K)$-equivariant.
  As mentioned above,  it only remains to prove  that the map $\LKtrace{AJ} :
  \operatorname{A}^{n+1}(X_{\bar K})[N] \to \LKtrace{J}^{2n+1}_a(X_\cx)[N]$ is
  $\operatorname{Gal}(\bar K/K)$-equivariant.
  
  For this, observe that the Bloch map $\lambda^{n+1} : 
  \operatorname{A}^{n+1}(X_{\bar K})[N] \longrightarrow
  H_{\text{\'et}}^{2n+1}(X_{\bar K}, \mmu_N^{\otimes (n+1)})$
  is Galois-equivariant, since it is constructed via natural maps of sheaves, all
  of which have natural Galois actions.  Moreover, on torsion the
  Bloch map factors through the Abel--Jacobi map. Indeed, over $\cx$,
  we have \cite[Prop.~3.7]{bloch79}
  $$
  \xymatrix@C=3em {
   \lambda^{n+1} :  \operatorname{A}^{n+1}(X_{\cx})[N] \ar[r]^<>(0.5){{AJ}[N]} &  
   J_{\cx}[N] \ar@{^(->}[r]& H_{\text{\'et}}^{2n+1}(X_{\cx}, \mmu_N^{\otimes
    (n+1)}).
  }
  $$
    Using rigidity for torsion cycles
\cite{lecomte86}, rigidity for torsion on abelian varieties, and
proper base change, we obtain the analogous statement over $\bar K$\,:
  $$
  \xymatrix@C=3em {
   \lambda^{n+1} :  \operatorname{A}^{n+1}(X_{\bar K})[N] \ar[r]^<>(0.5){\LKtrace
    {AJ}[N]} &  
   J_{\bar K}[N] \ar@{^(->}[r]& H_{\text{\'et}}^{2n+1}(X_{\bar K}, \mmu_N^{\otimes
    (n+1)}).
  }
  $$
  As described in \eqref{E:DescBigD},  the inclusion $J_{\bar K}[N]
  \hookrightarrow H_{\text{\'et}}^{2n+1}(X_{\bar
   K}, \mmu_N^{\otimes (n+1)})$ is also Galois-equivariant. By Lemma
  \ref{L:ez-emma}(b), we find that $
  \LKtrace{AJ}[N]$ is Galois-equivariant.
 \end{proof}

 \subsection{The Galois representation}
 We now conclude the proof of Theorem \ref{T:JacDesc} by constructing the
 correspondence $\Gamma\in \operatorname{CH}^{\dim (J)+n}(J\times_K X)$ inducing
 the desired morphism of Galois representations.  
 
 \begin{proof}[Proof of Theorem \ref{T:JacDesc}, Step 4\,: The Galois
  representation]
  Let $J/K$ be the model of $J^{2n+1}_a(X_{\mathbb C})$ from Step 2.
   (which was
  shown to be distinguished in Step 3\,; see Remark \ref{R:Distinguished}).   
  We will now construct a correspondence $\Gamma\in \operatorname{CH}^{\dim
   (J)+n}(J\times_K X)$ such that for each prime number $\ell$,  the correspondence
  $\Gamma$ induces an  inclusion of $\operatorname{Gal}(K)$-representations  
  $$
  \xymatrix{ 
   H^1(J_{\bar K},\mathbb Q_\ell) \ar@{^(->}[r]^<>(0.5){\Gamma_*}&
   H^{2n+1}(X_{\bar K},\mathbb Q_\ell(n)),
  }$$
  with image $\operatorname{N} ^n H^{2n+1}(X_{\bar K},\mathbb Q_\ell(n))$.

  Let $C$ and $\gamma \in \operatorname{CH}^{n+1}(C\times_K
  X)_{\mathbb{Q}}$ be the smooth, geometrically integral, pointed projective curve and the
  correspondence provided by Proposition \ref{P:niveauK}. As we have seen (in Steps~1 and 2 of the
  proof of Theorem \ref{T:JacDesc}), $\gamma$ induces a surjective homomorphism of
  complex abelian varieties $J(C_\cx) \to J_a^{2n+1}(X_\cx)$ that descends to a
  homomorphism $f : J(C) \to J$ of abelian varieties defined over $K$. Consider
  then the composite morphism
  \begin{equation}\label{E:cycles4-5'}
  \xymatrix{
   H^1(J_{\bar K},\mathbb Q_\ell) \ar@{^(->}[r]^<>(0.5){f^*}& H^1(J(C)_{\bar
    K},\mathbb Q_\ell)  \ar[r]^<>(0.5){\operatorname{alb}^*}_<>(0.5)\simeq
   &H^1(C_{\bar K},\mathbb Q_\ell) \ar[r]^<>(0.5){\gamma_*}&H^{2n+1}(X_{\bar
    K},\mathbb Q_\ell(n)),
  }
  \end{equation}
  where $\operatorname{alb} : C \to J(C)$ denotes the Albanese morphism induced by
  the $K$-point of $C$. This morphism is clearly injective and induced by a
  correspondence on $J\times_K X$, and we claim that its image is 
  $\operatorname{N} ^n H^{2n+1}(X_{\bar K},\mathbb
  Q_\ell(n))$. Indeed, the complexification of \eqref{E:cycles4-5'} together with
  the comparison isomorphisms yields a diagram
  \begin{equation}\label{E:cycles4-5''}
  \xymatrix{
   H^1(J_{a}^{2n+1}(X_\cx),\mathbb Q) \ar@{^(->}[r]^<>(0.5){(f_\cx)^*}&
   H^1(J(C)_{\cx},\mathbb Q) 
   \ar[r]^<>(0.5){(\operatorname{alb}_\cx)^*}_<>(0.5)\simeq &H^1(C_{\cx},\mathbb Q)
   \ar[r]^<>(0.5){(\gamma_\cx)_*}&H^{2n+1}(X_{\cx},\mathbb Q(n)),
  }
  \end{equation}
  where $(\operatorname{alb}_\cx)^*\circ (f_\cx)^*$ is easily seen to be the dual
  (via the natural choice of polarizations) of $(\gamma_\cx)_*$. Since the Hodge
  structure $H^1(C_{\cx},\mathbb Q)$ is polarized by the cup-product, we conclude
  by  \cite[Lemma 2.3]{ACMVdcg} that the image of \eqref{E:cycles4-5''} is equal
  to the image of $(\gamma_\cx)_*$, that is, to $\operatorname{N} ^n
  H^{2n+1}(X_{\cx},\mathbb
  Q(n))$.  Invoking the comparison isomorphism settles the claim.
  
  This completes the proof of Theorem \ref{T:JacDesc}.
 \end{proof}

 \section{Proof of Theorem \ref{T:main}\,: Part  II, regular homomorphisms and torsion points} \label{S:Tors-Gen} 
  
In order to upgrade Theorem \ref{T:JacDesc} to a statement
  about equivariance for arbitrary cycle classes, we reconsider and
  extend the theory of \emph{regular homomorphisms}. 
  Given a smooth projective complex variety $X$, a fundamental result of Griffiths
 \cite{griffiths} (and also \cite[p. 826]{griffiths68}) is that the Abel--Jacobi map   ${AJ}:\operatorname{A}^{n+1}(X)
 \longrightarrow J_a^{2n+1}(X)$ is a {regular homomorphism}. 
  This means that for every pair $(T,Z)$ with $T$ a pointed
 smooth integral complex
 variety,
 and
 $Z\in \operatorname{CH}^i(T\times X)$,  
 the composition 
 $$
 \begin{CD}
 T(\cx)@> w_Z >> \operatorname{A}^i(X)@>\phi >>J_a^{2n+1}(X)
 \end{CD}
 $$ 
 is induced by a morphism of complex varieties $\psi_Z:T\to J_a^{2n+1}(X)$,
 where, if $t_0\in T(\cx)$
 is
 the base point of $T$,  $w_Z:T(\cx)\to \operatorname{A}^i(X)$ is   given by
 $t\mapsto Z_t-Z_{t_0}$\,; here $Z_t$ is the refined Gysin fiber. Likewise, one
 defines regular homomorphisms for smooth projective varieties defined over an
 arbitrary algebraically closed field.
 We direct the reader to
 \cite[\S 3]{ACMVdcg} for a review of the material we use here on
 regular homomorphisms and \emph{algebraic representatives}, and to
 \cite[\S 4]{ACMVdcg} for the notion of a \emph{Galois-equivariant
  regular homomorphism}.   In this section we provide some results
regarding equivariance of regular homomorphisms\,; the main results are Propositions \ref{P:Tors-Gen'} and \ref{P:Tors-Gen}.

\subsection{Preliminaries}
We will  utilize the  following facts\,:
 
 \begin{pro}[{\cite[Thm.~2]{ACMV}}] \label{P:algcycles}
   Let $X/K$ be a  scheme of finite type   
  over a 
  perfect
  field~$K$.
  If $\alpha \in \operatorname{CH}^i(X_{\bar K})$
  is algebraically trivial, then there exist
  an abelian variety $A/K$,    a cycle $Z \in \operatorname{CH}^i( A\cross_K   
  X)$,  and a 
  $\bar K$-point $t\in 
  A(\bar K)$   such that $\alpha = Z_{t}-Z_{0}$.  \qed
 \end{pro}

 \begin{proof}
 We have shown  in \cite[Thm.~2]{ACMV} that 
 there exist
  an abelian variety $ A'/K$,  a cycle $  Z'$ on $ A'\cross_K 
  X$, and a pair of
  $\bar K $-points $t_1,t_0 \in  A'(\bar K)$   such that $\alpha = Z'_{t_1} -
 Z'_{t_0}$.  
  Let  $p_{13},p_{23}:  A'\cross_{K}   A' \cross_{K}  X
 \to    A' \cross_{K}   X$ be the obvious projections.
 Let ${Z}$ be defined as the cycle  
 $
 {Z}:= p_{13}^*Z'-p_{23}^*Z' 
 $ 
 on $ A'\times_{K}   A' \times_{K}  X$.
 For points $t_1,t_0\in  \underline A'({\bar K})$, 
  we have $
 {Z}_{(t_1,t_0)} = Z'_{t_1} - Z'_{t_0}$.  Thus setting $ A= A'\times_K A'$, we
 are done.
  \end{proof}

\begin{lem}\label{L:Tor-to-Tor}
Let $X$ be a scheme of finite type over an algebraically closed field $k$, and  let $A/k$ be an abelian variety.  
\begin{alphabetize}
\item Let  $Z\in \operatorname{CH}^i(A\times_k X)$.  The  
  map $w_Z : A(k) \to \operatorname{A}^i(X)$ is a homomorphism on torsion\,; more precisely, for each natural number $N$, $w_Z$ restricted to $A(k)[N]$ gives a homomorphism $w_Z[N]:A(k)[N]\to \operatorname{A}^i(X)[N]$.

\item Let $\phi:\operatorname{A}^i(X)\twoheadrightarrow A(k)$ be a surjective regular homomorphism.  There exists a natural number $r$ such that for any natural number $N$ coprime to $r$, $\phi$ is surjective on $N$-torsion\,; i.e., $\phi[N]:\operatorname{A}^i(X)[N]\twoheadrightarrow  A(k)[N]$ is surjective.  
\end{alphabetize}

\end{lem}

\begin{proof}
 (a) Since $w_Z$ factors as
  $A(k) \stackrel{\tau}{\longrightarrow} \operatorname{A}_0(A) \stackrel{Z_*}{\longrightarrow}
  \A^i(X)$,
  where $\tau(a):=[a]-[0]$, and  $Z_*$ is the group
  homomorphism induced by the action of the correspondence $Z$, it
  suffices to observe that  $\tau$ is a homomorphism on torsion.  
 In fact, $\tau$ is an isomorphism on torsion  
  \cite[Prop.~11, Lem.~p.259]{beauvillefourier} (which is based on  \cite[Thm.~(0.1)]{bloch76} and \cite{roitman80}).

(b) By \cite[Cor.~1.6.3]{murre83} (see also \cite[Lem.~4.9]{ACMVdcg}) there exists a $Z\in \operatorname{CH}^i(A\times_k X)$ so that the composition $\psi_Z:A(k) \stackrel{w_Z}{\longrightarrow} \A^i(X) \stackrel{\phi}{\longrightarrow}A(k)$ is induced by $r\cdot \operatorname{Id}_A$ for some integer $r$.  Let $N$ be any natural number coprime to $r$.
  Then $\psi_Z[N]$ is surjective, and therefore it  follows from (a) that $\phi[N]$ is surjective.
\end{proof}

\begin{rem}
Note that the proof of Lemma \ref{L:Tor-to-Tor}(b) actually shows that for all $N$,
we have a surjection $\A^i(X)[rN] \to A(k)[N]$. In particular, a surjective regular homomorphism $\phi:\operatorname{A}^i(X)\twoheadrightarrow A(k)$ (\emph{e.g.} the Abel--Jacobi map) induces a surjective homomorphism $\operatorname{A}^i(X)_{tors}\twoheadrightarrow A(k)_{tors}$ on torsion. 
\end{rem}

\subsection{Algebraically closed base change and equivariance of  regular homomorphisms}  
In this section we will utilize traces for algebraically closed field
extensions in arbitrary characteristic.  The main results of this
paper focus on the characteristic $0$ case, which we reviewed in
\S\ref{S:rigidity}.   The properties of the trace that  we utilize
here in positive characteristic are reviewed in \cite[\S
3.3.1]{ACMVdcg}\,;   the main difference is that we must potentially
keep track of some purely inseparable isogenies.   
 
  \begin{lem}\label{L:T-G-AC}
   Let $\Omega/k$ be an extension of algebraically closed fields, and let $X$ be a smooth projective variety over $k$. Let $A$ be an abelian variety over  $\Omega$ and let
   $\phi: \operatorname{A}^i(X_{\Omega})\to A(\Omega)$ be a surjective regular
   homomorphism.   Setting   $\tau : \LKtrace{A}_\Omega \to A$ to  be the $\Omega/k$-trace of $A$, 
 we have that $\tau$ is  a purely inseparable isogeny, which is an isomorphism in characteristic $0$.  Moreover, 
   there is a regular homomorphism $(\LKtrace
   \phi)_{\Omega}:\operatorname{A}^i(X_\Omega)\to \LKtrace A_\Omega
   (\Omega)$ making the following diagram commute
      \begin{equation} \label{E:comL}
 \vcenter{   \xymatrix@R=.5cm{
     \operatorname{A}^i(X_\Omega) \ar[r]^<>(0.5){(\LKtrace{\phi})_{\Omega}} \ar@{=}[d] & \LKtrace{A}_\Omega(\Omega) \ar[d]_{\simeq}^{\tau(\Omega)}\\
     \operatorname{A}^i(X_\Omega)  \ar[r]^<>(0.5){\phi}&
     A(\Omega). }}
   \end{equation} 
  \end{lem}

  \begin{proof}  Let us start by recalling some of the set-up from \cite[Thm.~3.7]{ACMVdcg}.  
First, consider the regular homomorphism $\LKtrace{\phi} : \operatorname{A}^i(X) \to \LKtrace{A}(k)$ constructed in Step 2 of the proof \cite[Thm.~3.7]{ACMVdcg}. 
   It fits into a commutative diagram  \cite[(3.9)]{ACMVdcg}
   \begin{equation}\label{E:com0}
    \vcenter{\xymatrix@R=.5cm{
     \operatorname{A}^i(X) \ar[r]^<>(0.5){\LKtrace{\phi}} \ar[dd] _{\text{ base change}}&
     \LKtrace{A}(k) \ar[d]^{\text{ base change}}\\
     & \LKtrace{A}_\Omega(\Omega) \ar[d]^{\tau(\Omega)}\\
     \operatorname{A}^i(X_\Omega)  \ar[r]^<>(0.5){\phi}&
     A(\Omega).
    }}
   \end{equation} 
   Since we are assuming that  $\phi: \operatorname{A}^i(X_{\Omega})\to A(\Omega)$ is surjective, Step 3 of the proof of \cite[Thm.~3.7]{ACMVdcg} yields that $\LKtrace{\phi} : \operatorname{A}^i(X) \to \LKtrace{A}(k)$ is surjective, and that $\tau : \LKtrace{A}_\Omega \rightarrow A$ is a purely inseparable isogeny, and thus an isomorphism in characteristic $0$.    In particular, $\tau(\Omega) : \LKtrace{A}_\Omega(\Omega) \rightarrow A(\Omega)$ is an isomorphism.

   Now consider the regular homomorphism $\LKtrace{\phi}_\Omega : \operatorname{A}^i(X_{\Omega})\to \LKtrace{A}_\Omega(\Omega)$ constructed in Step 1 of the proof of \cite[Thm.~3.7]{ACMVdcg}, which by \emph{loc.~cit.}~is surjective.   We can therefore fill in diagram \eqref{E:com0} to obtain\,:
   \begin{equation} \label{E:com}
   \vcenter{ \xymatrix@R=.5cm{
     \operatorname{A}^i(X) \ar[r]^<>(0.5){\LKtrace{\phi}} \ar[d]_{\text{ base change}}&
     \LKtrace{A}(k) \ar[d]^{\text{ base change}}\\
     \operatorname{A}^i(X_\Omega) \ar[r]^<>(0.5){(\LKtrace{\phi})_{\Omega}} \ar@{=}[d] & \LKtrace{A}_\Omega(\Omega) \ar[d]_{\simeq}^{\tau(\Omega)}\\
     \operatorname{A}^i(X_\Omega)  \ar[r]^<>(0.5){\phi}&
     A(\Omega). }}
   \end{equation} 
 We claim  that \eqref{E:com} is commutative.    
To start, the commutativity of the top square is established in Step 1
of the proof of   \cite[Thm.~3.7]{ACMVdcg}, and we have already
confirmed the commutativity of the outer rectangle, above.   For the
bottom square we argue as follows.  

By rigidity for torsion cycles on $X$ (\cite{jannsen15,lecomte86}\,; see also
\cite[Thm.~3.8(b)]{ACMVdcg}) and for torsion points on $\LKtrace A$,
the vertical arrows in diagram \eqref{E:com} are isomorphisms on
torsion. 
 A little more naively (i.e., without using \cite{jannsen15}), one can simply fix a prime number $\ell$ not equal to  $\operatorname{char}k$, and consider torsion to be  $\ell$-power torsion, and the rest of the argument goes through without change.  
 The top square and outer rectangle are commutative, and thus
\eqref{E:com} is commutative on torsion.   
Now let $\alpha\in
\operatorname{A}^i(X_\Omega)$.  By Weil \cite[Lem.~9]{weil54} (e.g.,
Proposition \ref{P:algcycles}) there exist an abelian variety $B/\Omega$,
a cycle class $Z\in \operatorname{CH}^i(B\times_\Omega X_\Omega)$, and
an $\Omega$-point $t\in B(\Omega)$ such that $\alpha=Z_t-Z_0$.  Then
consider the following diagram (not \emph{a priori} commutative)\,:
        \begin{equation}\label{E:com1}
    \vcenter{\xymatrix@R=.5cm{
  B(\Omega) \ar[r]^{w_Z} \ar@{=}[d]&    \operatorname{A}^i(X_\Omega) \ar[r]^<>(0.5){(\LKtrace{\phi})_{\Omega}} \ar@{=}[d] & \LKtrace{A}_\Omega(\Omega) \ar[d]_{\simeq}^{\tau(\Omega)}\\
  B(\Omega) \ar[r]^{w_Z}&   \operatorname{A}^i(X_\Omega)  \ar[r]^<>(0.5){\phi}&
     A(\Omega). }}
   \end{equation} 
   The left-hand square is obviously commutative.  We have shown that
   the right-hand square is commutative on torsion.
The horizontal arrows on  the left send torsion points to torsion cycle classes (Lemma
   \ref{L:Tor-to-Tor}(a)).  Therefore the whole diagram \eqref{E:com1} is
   commutative on torsion.  Since torsion points are Zariski
   dense in abelian varieties,
 the diagram
   is commutative if we replace $\operatorname{A}^i(X_\Omega)$ with
   $\operatorname{Im}(w_Z)$.  Since $\alpha\in
   \operatorname{Im}(w_Z)$, we see that $(\tau(\Omega)\circ (\LKtrace
   \phi)_\Omega)(\alpha)=\phi(\alpha)$.  Thus, since $\alpha$ was
   arbitrary, the lemma is proved.
  \end{proof}

  \begin{pro}\label{P:Tors-Gen'}
   Let $\Omega/k$ be an extension of algebraically closed fields of characteristic $0$, and let $X$ be a smooth projective variety over $k$. Let $A$ be an abelian variety over  $\Omega$ and let
   $\phi: \operatorname{A}^i(X_{\Omega})\to A(\Omega)$ be a surjective regular
   homomorphism.    Then $A$ admits a model over $k$, the $\Omega/k$-trace of $A$,  such that $\phi$ is
   $\operatorname{Aut}(\Omega/k)$-equivariant.
  \end{pro}

  \begin{proof}  
This follows directly from Lemma  \ref{L:T-G-AC}.  Indeed, by the construction of $(\LKtrace \phi)_\Omega$ in Step 1 of  \cite[Thm.~3.7]{ACMVdcg},  $(\LKtrace \phi)_\Omega$ is $\operatorname{Aut}(\Omega/k)$-equivariant.  Then, since $\tau: \LKtrace A_\Omega \to A$ is an isomorphism, we are done. 
  \end{proof}

 \begin{rem}
More generally, if $\operatorname{char} k\ne 0$, then in the notation of Proposition \ref{P:Tors-Gen'}, the abelian variety $A$ admits a purely inseparable isogeny to an abelian variety over $\Omega$ that descends to $k$, namely the $\Omega/k$-trace.  Moreover, under this purely inseparable isogeny, the $\Omega$-points of both abelian varieties are identified, and under the induced action of $\operatorname{Aut}(\Omega/k)$ on $A(\Omega)$, we have that $\phi$ is $\operatorname{Aut}(\Omega/k)$-equivariant.  
\end{rem}

 \subsection{Galois-equivariant regular homomorphisms and torsion points}
 The main point of this subsection is to prove Proposition \ref{P:Tors-Gen}.  
 This allows us to  utilize results of \cite{ACMVdcg} on regular homomorphisms in
 the setting of torsion points.  
 We start with the following lemma.

 \begin{lem}\label{L:key}
  Let $A$ be an abelian variety over a perfect field $K$ and let
  $\phi: \operatorname{A}^i(X_{\bar K})\to A(\bar K)$ be a regular
  homomorphism.  Assume that there is a prime   $\ell \not 
  =\operatorname{char}(K)$ such that for all positive integers $n$
  we have that the map 
  $\phi[\ell^n]: \operatorname{A}^i(X_{\bar K})[\ell^n]\to A[\ell^n]$ is
  $\operatorname{Gal}(K)$-equivariant.  Let $B/K$ be an abelian variety and let 
  $Z \in \operatorname{CH}^i(B\times_KX)$ be a cycle class. Then the induced morphism
  $\psi_{Z_{\bar K}}: B_{\bar K} \to A_{\bar K}$ 
  is defined over $K$.  
 \end{lem}
 \begin{proof} 
  Since (geometric) $\ell$-primary torsion points are Zariski dense in the graph
  of
  $\psi_{Z_{\bar K}}$ inside $B\times_K A$, it
  suffices to show that the induced morphism $B(\bar K) \to A(\bar K)$
  is Galois-equivariant on $\ell$-primary torsion.  
  Since the map $w_Z : B(\bar K) \to
  \operatorname{A}^i(X_{\bar K})$ is Galois-equivariant and since
  $\phi: \operatorname{A}^i(X_{\bar K})\to A(\bar K)$ is
  Galois-equivariant on $\ell^n$-torsion for all positive integers $n$, it is
  even enough to show that the
  map $w_Z : B(\bar K) \to \operatorname{A}^i(X_{\bar K})$ sends
  torsion points of $B(\bar K)$ to torsion cycles in
  $\operatorname{A}^i(X_{\bar K})$.  This is Lemma \ref{L:Tor-to-Tor}(a).
 \end{proof}

 We can now prove\,: 
 
 \begin{pro}\label{P:Tors-Gen}
  Let $A$ be an abelian variety over a perfect field $K$ and let
  $\phi: \operatorname{A}^i(X_{\bar K})\to A(\bar K)$ be a regular
  homomorphism.  Assume that there is a prime   $\ell \not 
  =\operatorname{char}(K)$ such that for all positive integers $n$ 
  the map 
  $\phi[\ell^n]: \operatorname{A}^i(X_{\bar K})[\ell^n]\to A[\ell^n]$ is
  $\operatorname{Gal}(K)$-equivariant.  Then $\phi$ is
  $\operatorname{Gal}(K)$-equivariant.
 \end{pro}

 \begin{proof}
  Let $\alpha \in \operatorname{A}^i(X_{\bar K})$, and let $\sigma\in
  \operatorname{Gal}(K)$.  From Proposition \ref{P:algcycles}, we have
  an abelian variety $B/K$,    a cycle $Z \in \operatorname{CH}^i( B\cross_K   
  X)$,  and a 
  $\bar K$-point $t\in 
  B(\bar K)$   such that $\alpha = Z_{t}-Z_{0}$. 
  Now consider the following diagram (not \emph{a  priori} commutative)\,:
  $$
  \begin{CD}
  B(\bar K)@>w_{Z_{\bar K}} >> \operatorname{A}^i(X_{\bar K}) @>\phi>> A(\bar K)\\
  @V\sigma_B^*VV @V\sigma_X^*VV @V\sigma_A^*VV\\
  B(\bar K)@>w_{Z_{\bar K}} >> \operatorname{A}^i(X_{\bar K}) @>\phi>> A(\bar
  K).\\
  \end{CD}
  $$
  Since $Z$ is defined over $K$, and the base point $0$ is defined over $K$, the
  left-hand square is commutative (e.g., \cite[Rem.~4.3]{ACMVdcg}).  
  It follows from  Lemma \ref{L:key} that the outer rectangle is also commutative.
  Therefore, from Lemma \ref{L:ez-emma}(a), the right-hand square in the diagram
  is commutative on the image of $w_{Z_{\bar K}}$.   
  In particular,  $\phi(\sigma_X^*\alpha)=\sigma_B^*\phi(\alpha)$. 
 \end{proof}

 \section{Proof of Theorem \ref{T:main}\,: Part  III, the coniveau filtration is
  split} \label{S:Andre}
 We now complete the proof of Theorem \ref{T:main} by showing that the coniveau
 filtration is split (Corollary \ref{C:CNFSplit}). For this purpose, we use Yves
 Andr\'e's theory of motivated cycles \cite{AndreIHES}. Along the way, we show in
 Theorem \ref{T:mot} that the existence of a phantom isogeny class for
 $\coniveau^nH^{2n+1}(X_{\bar K},\rat_\ell(n))$ for all primes $\ell$
 follows directly from Andr\'e's theory. Note that we already  proved this in
 Theorem \ref{T:JacDesc}  in a more precise form, namely by showing that there
 exists a \emph{distinguished} phantom abelian variety within the isogeny class.

 \vskip .2 cm 
 For clarity, we briefly review  the setup of Andr\'e's theory of  motivated
 cycles, and fix some notation.  
 Given a smooth projective variety $X$ over a field $K$ and a prime $\ell \neq
 \operatorname{char}(K)$, let us denote $\operatorname{B}^j(X)_\rat$ the image
 of
 the cycle class map $\operatorname{CH}^j(X)_\rat \to H^{2j}(X_{\bar K},
 \rat_\ell(j))$.
 A \emph{motivated cycle} on
 $X$ with rational coefficients is an element of the graded algebra 
 $\bigoplus_r
 H^{2r}(X_{\bar K}, \rat_\ell(r))$ of the form $\mathrm{pr}_* (\alpha \cup *
 \beta)$, where $\alpha$ and $\beta$ are elements of
 $\operatorname{B}^*(X\times_K Y)_\rat$ with $Y$ an arbitrary smooth projective
 variety over $K$, $\mathrm{pr} : X\times_K Y \to X$ is the natural projection,
 and $*$ is the Lefschetz involution on $\bigoplus_r H^{2r}((X\times_K Y)_{\bar
  K}, \rat_\ell(r))$ relative to any polarization on $X\times_K Y$.  
 The set of
 motivated cycles on $X$, denoted
 $\operatorname{B}^\bullet_{\mathrm{mot}}(X)_\rat$,
 forms a  graded $\rat$-sub-algebra of $\bigoplus_r H^{2r}(X_{\bar K},
 \rat_\ell(r))$, with $\operatorname{B}^r_{\mathrm{mot}}(X)_\rat\subseteq
 H^{2r}(X_{\bar K},
 \rat_\ell(r))$\,; \emph{cf.} \cite[Prop. 2.1]{AndreIHES}. 
 Taking $Y=\operatorname{Spec}K$ above, we have an inclusion 
 $\operatorname{B}^r(X)_{\mathbb Q}\subseteq
 \operatorname{B}^r_{\mathrm{mot}}(X)_\rat$. 
 Moreover there is a
 notion of motivated correspondences between smooth projective varieties, and
 there is a
 composition law with the expected properties.

 \begin{pro}\label{P:mot}
  Let $Y$ and $X$ be smooth projective varieties over a field $K\subseteq \mathbb
  C$. Consider
  a   motivated cycle  $\gamma \in
  \operatorname{B}_{\operatorname{mot}}^{d_Y+r}(Y\times_K X)_\rat$ and its action
  $$\xymatrix{\gamma_* : H^j(Y_{\bar K}, \rat_\ell) \ar[r]& H^{j+2r}(X_{\bar K},
  \rat_\ell(r))}.$$ 
Then $\operatorname{Im}(\gamma_*)$
  (resp.~$\operatorname{ker}(\gamma_*)$) is a direct summand of the
  $\operatorname{Gal}(K)$-representation $H^{j+2r}(X_{\bar K}, \rat_\ell(r))$
  (resp. $H^j(Y_{\bar K}, \rat_\ell)$).  
 \end{pro}

 \begin{proof}
  We are going to show   that if  $\gamma \in
  \operatorname{B}^{d_Y+r}_{\mathrm{mot}}(Y\times_K X)_\rat$ is a motivated
  correspondence, then there exists an
  idempotent motivated correspondence $p \in
  \operatorname{B}^{d_Y}_{\mathrm{mot}}(Y\times_K Y)_\rat$ such that 
  $p_*H^j(Y_{\bar K}, \rat_\ell)  = \operatorname{ker}(\gamma_*)$.   Assuming the
  existence of such a $p$, this would establish that
  $\operatorname{ker}(\gamma_*)$ is a direct summand of $H^j(Y_{\bar K},
  \rat_\ell)$ as a $\mathbb Q_\ell$-vector space.  But then by 
  \cite[Scolie 2.5]{AndreIHES}, motivated cycles on a smooth projective variety
  $Y$
  over $K$ are exactly the 
  $\operatorname{Gal}(K)$-invariant motivated cycles on $Y_{\bar K}$\,;
  therefore $\operatorname{ker}(\gamma_*)$
  is indeed a direct summand of $H^j(Y_{\bar K}, \rat_\ell)$ as a
  $\operatorname{Gal}(K)$-representation, completing the proof. 
  The statement
  about the image of $\gamma_*$ 
  follows by duality.

  The existence of $p$ follows formally from
  \cite[Thm.~0.4]{AndreIHES}\,: the 
  $\otimes$-category of pure motives $\mathscr{M}$ over a field $K$ of
  characteristic zero obtained
  by using motivated correspondences rather than algebraic correspondences is a
  graded, abelian semi-simple, polarized, and Tannakian category over $\rat$.
  Indeed, using the notations from \cite[\S 4]{AndreIHES} and viewing $\gamma$ as
  a
  morphism from the motive $\mathfrak{h}(Y)$ to the motive $\mathfrak{h}(X)(r)$,
  we see by semi-simplicity that there exists an idempotent motivated
  correspondence $p \in  \operatorname{B}^{d_Y}_{\mathrm{mot}}(Y\times_K
  Y)_\rat$
  such that $\operatorname{ker}(\gamma) = p\mathfrak{h}(Y)$.  Now the Tannakian
  category $\mathscr{M}$ is neutralized by the fiber functor to the category of
  $\rat_\ell$-vector spaces given by the $\ell$-adic realization functor. Since
  by
  definition a fiber functor is exact, $p_*H^j(Y_{\bar K},
  \rat_\ell) 
  = \operatorname{ker}(\gamma_*)$ as $\rat_\ell$-vector spaces. 
 \end{proof}
 
 \begin{teo}\label{T:mot} Suppose $X$ is a smooth projective variety over a
  field $K\subseteq \mathbb C$, and let $n$ be a nonnegative integer.
  The $\operatorname{Gal}(K)$-representation $\coniveau^nH^{2n+1}(X_{\bar
   K},\rat_\ell(n))$ admits a phantom\,; more precisely there exist an abelian
  variety $J'$ over $K$ and a 
  correspondence $\Gamma' \in
  \operatorname{CH}^{\dim{J'}+n}(J'\times_K X)$ such  that the morphism of
  $\operatorname{Gal}(K)$-representations 
  \begin{equation}\label{E:splitinclusion}
 \xymatrix{ \Gamma'_* : H^1(J'_{\bar K},\rat_\ell) \ar@{^(->}[r]&
  H^{2n+1}(X_{\bar
   K},\rat_\ell(n))}
  \end{equation}
  is split injective with image $\coniveau^nH^{2n+1}(X_{\bar
   K},\rat_\ell(n))$.
 \end{teo}
 \begin{proof}
  Let $C$ and $\gamma \in \operatorname{CH}^{n+1}(C\times_K
  X)_{\mathbb{Q}}$ be the pointed curve and the correspondence provided by
  Proposition \ref{P:niveauK}.
  By Proposition \ref{P:mot} and its proof, there is an idempotent motivated
  correspondence $q \in \operatorname{B}^{1}_{\mathrm{mot}}(C\times_K C)_\rat$ 
  such that 
  $q_*H^1(C_{\bar K}, \rat_\ell)  \stackrel{\gamma_*}{\longrightarrow}
  H^{2n+1}(X_{\bar
   K},\rat_\ell(n))$ is a monomorphism of
  $\operatorname{Gal}(K)$-representations with image $
  \coniveau^nH^{2n+1}(X_{\bar
   K},\rat_\ell(n))$, which is itself a direct summand of $H^{2n+1}(X_{\bar
   K},\rat_\ell(n))$.

  Now we claim that for smooth projective
  varieties defined over a field of characteristic zero,
  we have 
  $\operatorname{B}^{1}_{\mathrm{mot}}(-)_\rat = \operatorname{B}^{1}(-)_\rat$. 
  Over  an algebraically closed field of characteristic zero this
  is a consequence of the Lefschetz $(1,1)$-theorem. Over a field $K$ of
  characteristic zero, the claim follows from the following two facts\,: (1) if
  $Y$ is a smooth projective variety over $K$, then 
  $\operatorname{B}^{r}(Y)_\rat $ consists of the
  $\operatorname{Gal}(K)$-invariant classes in $\operatorname{B}^{r}(Y_{\bar
   K})_\rat$ by a standard norm argument, and   similarly (2)
  $\operatorname{B}_{\mathrm{mot}}^{r}(Y)_\rat $ consists of the
  $\operatorname{Gal}(K)$-invariant classes in
  $\operatorname{B}_{\mathrm{mot}}^{r}(Y_{\bar
   K})_\rat$ by \cite[Scolie 2.5]{AndreIHES}.

  Therefore the motivated idempotent $q$ is in fact an idempotent correspondence
  in $\operatorname{B}^{1}(C\times_K C)_\rat$, and
  thus
  defines, up to isogeny, an idempotent endomorphism of $\pic^\circ(C)$.
  Its image $J'$, which is only defined up to isogeny, 
  is the sought-after abelian variety such that $q_*H^1(C_{\bar K}, \rat_\ell)
  \cong H^1(J'_{\bar K}, \rat_\ell)$. Composing the transpose of the graph of the
  morphism $C \hookrightarrow \pic^\circ(C) \twoheadrightarrow J'$ with
  the algebraic correspondence $\gamma$ yields the desired correspondence $\Gamma'
  \in \operatorname{CH}^{\dim J'+n}(J'\times_K X)$. 
 \end{proof}
 
 \begin{rem}
  The main difference with \cite[Thm.~2.1]{ACMVdcg} is that we do not know if the 
  splitting in Theorem \ref{T:mot}
  is
  induced by an algebraic correspondence over $K$. In that respect
  \cite[Thm.~2.1]{ACMVdcg}
  is more precise. \end{rem}

 A nice consequence of Proposition \ref{P:mot}  is the
 following\,:

 \begin{cor}\label{C:CNFSplit}
  Let $X$ be a smooth projective variety over a field $K\subseteq \mathbb C$.
  The geometric coniveau filtration on the $\operatorname{Gal}(K)$-representation
  $H^n(X_{\bar K},\rat_\ell)$ is split.
 \end{cor}

 \begin{proof}
  Let $r$ be a  nonnegative integer. 
  Using the coniveau hypothesis,  resolution of singularities, mixed Hodge
  theory, and comparison isomorphisms, there exist a smooth
  projective variety
  $Y$ of dimension $\dim X-r$ over $K$
  and a morphism $f : Y \to X$  such that the induced morphism of
  $\operatorname{Gal}(K)$-representations
  $$
   f_*: H^{n-2r}(Y_{\overline K},{\mathbb{Q}}_\ell(-r)) \rightarrow
   H^{n}(X_{\overline K},{\mathbb{Q}}_\ell)
  $$ 
  has image
  $\coniveau^rH^{n}(X_{\overline K},{\mathbb{Q}}_\ell)$\,; see e.g.
  \cite[Sec. 4.4(d)]{illusiemiscellany}. 
  The splitting of the coniveau filtration
  follows from Proposition \ref{P:mot} and the Krull--Schmidt theorem. 
 \end{proof}
 
  \begin{proof}[Proof of Theorem \ref{T:main}]
  Everything except for the splitting of the inclusion \eqref{E:Tmain-1}  in
  Theorem \ref{T:main} is shown by combining  Theorem \ref{T:JacDesc} with  Propositions \ref{P:Tors-Gen'} and \ref{P:Tors-Gen}.  The splitting follows
  from Corollary \ref{C:CNFSplit}.   
 \end{proof}

 \section{A functoriality statement}

 Recall that if $X$ and $Y$ are  smooth projective varieties over a field
 $K\subseteq \cx$, and we are given   a
 correspondence $Z\in \operatorname{CH}^{m-n+\dim X}(X\times_K Y)$, then $Z$ induces functorially a homomorphism of complex abelian
 varieties $$\psi_{Z_\cx} : J^{2n+1}_a(X_\cx) \to J^{2m+1}_a(Y_\cx)$$
 that is compatible with the Abel--Jacobi map.
 
 \begin{pro}
 	\label{P:functoriality}
 	Denote $J$ and $J'$ the distinguished models of $ J^{2n+1}_a(X_\cx) $ and $
 	J^{2m+1}_a(Y_\cx)$ over $K$. Then
 	the homomorphism $\psi_{Z_\cx}$ descends to a $K$-homomorphism of abelian
 	varieties $\psi_Z : J \to J'$. 
 	
 	In particular, given a morphism $f:X\to Y$ defined over $K$,  the graph of $f$ and its transpose induce homomorphisms
 	$$f_* : J^{2n+1}_{X/K}\to J^{2(n-\dim X+\dim Y)+1}_{Y/K} \quad \text{and}\quad f^* : J^{2n+1}_{Y/K}\to J^{2n+1}_{X/K}.$$ 
 	This makes our descent functorial for morphisms of smooth projective varieties over $K$. 
 \end{pro}
 \begin{proof}
 	By Theorem \ref{T:main}, the
 	Abel--Jacobi map $AJ:\operatorname{A}^{n+1}(X_\cx) \to J_a^{2n+1}(X_\cx)$ is
 	$\aut(\cx/K)$-equivariant. Applying Lemma \ref{L:ez-emma}(a) to the commutative
 	square
 	$$\xymatrix@C=3.5em {\operatorname{A}^{n+1}(X_\cx) \ar@{->>}[r]^{AJ}
 		\ar[d]_{(Z_{\cx})_*} & J_a^{2n+1}(X_\cx) \ar[d]^{\psi_{Z_\cx}}\\
 		\operatorname{A}^{m+1}(Y_\cx) \ar@{->>}[r]^{AJ}  & J_a^{2m+1}(Y_\cx)
 	}$$
 	shows that $\psi_{Z_\cx}$ is $\aut(\cx/K)$-equivariant.   From the theory of the $\mathbb C/\bar K$-trace, $\psi_{Z_\cx}$
 	descends to a morphism $\LKtrace \psi_{Z_\cx}: \LKtrace{J}^{2n+1}_a(X_{\mathbb
 		C})\to \LKtrace J_a^{2m+1}(Y_\cx)$ over  $\bar K$.  Then the 
 	$\aut(\cx/K)$-equivariance of $\psi_{Z_\cx}$ on $\mathbb C$-points implies
 	$\LKtrace \psi_{Z_\cx}$ is $\operatorname{Gal}(\bar K/K)$-equivariant on $\bar
 	K$-points, and so  descends from $\bar K$ to $K$. Alternately, $\psi_{Z_\cx}$ descends to $K$ simply by $\cx/K$-descent. 
	 \end{proof}
 
 \begin{rem}
Proposition \ref{P:functoriality} could have been proved earlier by using Theorem \ref{T:JacDesc}, together with the fact (see Lemma \ref{L:Tor-to-Tor}(b)) that $AJ[N]:\operatorname{A}^{n+1}(X_{\mathbb C})[N]\to
J_a^{2n+1}(X_{\mathbb C})[N]$ is surjective for all $N$ not divisible by a finite number of fixed primes and the fact that torsion points on an abelian variety of order  not divisible by a finite number of fixed primes are dense.
 \end{rem}

 \section{Deligne's theorem on complete intersections of
  Hodge level $1$} \label{S:Deligne}
 
 We recapture Deligne's result \cite{deligneniveau} on intermediate Jacobians of
 complete
 intersections of Hodge level $1$ (Deligne's primary motivation was to establish
 the Weil conjectures for those varieties\,; of course Deligne established the
 Weil conjectures in full generality a few years later)\,:
 
 \begin{teo}[Deligne \cite{deligneniveau}]
  Let $X$ be a smooth complete intersection of odd dimension $2n+1$ over a field
  $K\subseteq \cx$. Assume that $X$ has Hodge level $1$, that is, assume that
  $h^{p,q}(X_\cx) = 0$ for all $|p-q| >1$. Then the intermediate Jacobian
  $J^{2n+1}(X_\cx)$ is a complex  abelian variety that  is defined over $K$.
 \end{teo}
 
 \begin{proof} First note that the assumption that $X$ has Hodge level 1 implies
  that the cup product on $H^{2n+1}(X_\cx,\mathbb Z)$ endows the complex torus
  $J^{2n+1}(X_\cx)$ with a Riemannian form so that $J^{2n+1}(X_\cx)$ is naturally a
  principally polarized complex abelian variety. 
  Deligne's proof that this complex abelian variety is defined over $K$ uses the
  irreducibility of the monodromy action of the fundamental group of the universal
  deformation of $X$  on $H^{2n+1}(X_\cx,\rat)$ and on
  $H^{2n+1}(X_\cx,\mathbb Z / \ell)$ for all primes $\ell$. Here, we give an
  alternate proof based on our Theorem \ref{T:main}.
  
  Denote $V_m(a_1,\ldots,a_k)$ a smooth complete intersection of dimension $n$ of
  multi-degree $(a_1,\ldots,a_k)$ inside $\mathbb P^{m+k}$. A complete
  intersection $X$ of
  Hodge level $1$ of odd dimension is of one of the following types\,:
  $V_{2n+1}(2), V_{2n+1}(2,2), V_{2n+1}(2,2,2), V_3(3), V_3(2,3), V_5(3),
  V_3(4)$\,; see for instance \cite[Table 1]{rapoport}.
  In the cases where $X$ is one of the above and $X$ has dimension $3$, then $X$
  is Fano and as such is rationally connected, and therefore
  $\operatorname{CH}_0(X_\cx) = \mathbb Z$. In all of the other listed
  cases, 
  it is known \cite[Cor.~1]{otw} that $\operatorname{CH}_0(X_\cx)_\rat,\ldots,
  \operatorname{CH}_{n-1}(X_\cx)_\rat$ are spanned by linear sections. 
  By  \cite[Thm.~3.2]{esnaultlevine}, which is based on a decomposition of the diagonal argument \cite{BlSr83}, it follows that if $X$ is
  a complete intersection of Hodge level 1, then the
  Abel--Jacobi map $\operatorname{A}^n(X_\cx) \to J^{2n+1}(X_\cx)$ is surjective,
  i.e., that $J^{2n+1}(V_{\mathbb C}) = J^{2n+1}_a(V_{\mathbb C})$. 
  Theorem \ref{T:JacDesc} implies that the complex abelian
  variety $J^{2n+1}(X_{\mathbb C})$ has a distinguished model over $K$. 
 \end{proof}

 \section{Albaneses of Hilbert schemes} \label{S:Hilb}

  Over the complex numbers the image of the Abel--Jacobi map is dominated by
  Albaneses of resolutions of singularities of products of irreducible components
  of Hilbert  
  schemes.  
  Since Hilbert schemes are functorial, and in particular  defined over the field of definition, and
  since the image of the Abel--Jacobi map descends to the field of definition,  one might expect this abelian variety to be dominated by   Albaneses of resolutions of singularities of  products of irreducible components of Hilbert  schemes defined over the field of definition.
 In this section, we show this is the case, thereby proving 
 Theorem  
  \ref{Ta:MazQ1}.   Our  approach  utilizes the theory of Galois equivariant regular homomorphisms, and consequently, we obtain some related results over perfect fields in  arbitrary characteristic.  
  
    \medskip

 \subsection{Regular homomorphisms  and difference maps} \label{S:DiffMaps}
 In this section we give an equivalent theory of regular homomorphisms  and
 algebraic
 representatives that  does not rely on pointed varieties.  
 
 Let $X/k$ be a smooth projective variety over the algebraically closed
 field $k$, let
 $T/k$ be a smooth
 integral   
 variety and let $Z$ be a codimension-$i$ cycle on $T \times_{k}X$.
 Let  $p_{13},p_{23}:  T\cross_{k} T \cross_{k}  X
 \to   T\cross_{k}   X$ be the obvious projections.
 Let $\widetilde {
  Z}$ be defined as the cycle  
 \[
 \widetilde {Z}:= p_{13}^*Z-p_{23}^*Z 
 \] 
 on $T\times_{k}  T\times_{k}  X$.
 For points $t_1,t_0\in  T({k})$, 
 we have $
 \til {Z}_{(t_1,t_0)} = Z_{t_1} - Z_{t_0}$.  
 We therefore have a map  
 \begin{equation}
 \label{E:defdiffk}
 \xymatrix@R=.1cm{
  (T\cross_{k} T)({k}) \ar[r]^<>(0.5){y_{Z}} & \A^i(X) \\
  (t_1,t_0) \ar@{|->}[r]&  Z_{t_1} - Z_{t_0}.
 }
 \end{equation}

 \begin{lem} \label{L:DiffReg}
  Let $X/k$ be a smooth projective variety over an algebraically closed field
  ${k}$, and let $A/k$ be an abelian variety.
  A homomorphism of groups
  $
  \phi:
  \operatorname{A}^i(X)\to  A(k)
  $
  is regular if and only if for every pair $(T,  Z)$ with $T$   a
  smooth integral  
  variety over ${k}$
  and
  $Z\in \operatorname{CH}^i( T\times_{k}  X)$,  
  the composition 
  $$
  \begin{CD}
  (T\times_{k}T)({k})@> y_{Z} >> \operatorname{A}^i(X)@>\phi >>A({k})
  \end{CD}
  $$ 
  is induced by a morphism of varieties $\xi_{ Z}:T\times_{k} T\to A$.
 \end{lem}
 
 \begin{proof}
   If $\phi:\A^i(X) \to A(k)$ is a regular homomorphism to an
   abelian variety, then $\phi\circ y_{ Z}$ is induced by a morphism
   of varieties $T\cross_{k} T \to A$\,; indeed after choosing
   any diagonal base point
   $(t_0,t_0)\in (T\times_{k} T)(k)$, the maps
   $\phi\circ y_{ Z}$ and $\phi\circ w_{\widetilde Z,(t_0,t_0)}$
   agree.  Conversely, suppose $\phi\circ y_{ Z}$ is induced by a
   morphism $\xi_{Z}$ of varieties, and let $t_0 \in T({k})$ be
   any base point.  Let $\iota$ be the inclusion
   $\iota: T \to T\times \st{t_0} \subset T\times T$.  Then
   $w_{Z,t_0} = y_{Z}|_{\iota(T)}$, and $\phi\circ w_Z$ is
   induced by the morphism $\xi_{ Z} \circ \iota$.
 \end{proof}

 Now suppose that $X$ is a smooth projective variety over $K$, that $T$ is a smooth integral 
 quasi-projective variety over $K$, and that $Z$ is a codimension-$i$ cycle on
 $T\times_K X$.  The cycle $\til Z = p_{13}^*Z -p_{23}^*Z$ on $T\times_K T \times_K X$
 is again defined over $K$.
 
 \begin{lem}
  \label{L:DiffDesc}
  Let $K$ be a perfect field, suppose $X$, $Z$ and $T$ are as above, and let $A/K$ be an abelian variety.  If
  $\phi:
  \A^i(X_{\bar K}) \to A(\bar K)$ is a $\gal(K)$-equivariant
  regular homomorphism, then the induced morphism $\xi_{Z_{\bar K}}:
  (T\times_K T)_{\bar K} \to A_{\bar K}$ is also $ \gal(K)$-equivariant on $\bar
  K$-points, and thus $\xi_{Z_{\bar K}}$ descends to a morphism
  $\xi_Z: T \times_K T \to A$ of varieties over $K$.
 \end{lem}
 
 \begin{proof}
  For each $\sigma \in \gal(K)$
there is a commutative diagram
  (see \cite[Rem.~4.3]{ACMVdcg})
  $$
  \xymatrix@C=1.5cm@R=.75cm{
   (T\times_K T)({\bar K}) \ar@{->}[r]^{y_{ Z_{\bar K}}} 
   \ar@{->}[d]^{\sigma_{T\times
     T}^*} &   \operatorname{A}^i(X_{\bar K}) \ar@{->}[r]^{\phi}
   \ar@{->}[d]^{\sigma_X^{*}}  &
   A({\bar K}) \ar@{->}[d]^{\sigma_{A}^*}\\
   (T\times_K T)({\bar K})  \ar@{->}[r]^{y_{Z_{\bar K}}}
   &\operatorname{A}^i(X_{\bar K})
   \ar@{->}[r]^{\phi}& A({\bar K}). \\
  }
  $$
  Now $\phi$ is $\gal(K)$-equivariant by hypothesis, and
  $y_{Z_{\bar K}}$
  is $\gal(K)$-equivariant since $\til Z$, $T$ and $X$ are
  defined over $K$. Consequently, $\xi_{Z_{\bar K}}$ is
  $\gal(K)$-equivariant, as claimed.
 \end{proof}

 \subsection{Albaneses of Hilbert schemes   and  the Abel--Jacobi map}
 
 We are now in a position to prove the following theorem, which will allow us to
 prove  Theorem \ref{Ta:MazQ1}.

 \begin{teo}\label{T:MazQ1}
  Suppose $X$ is a smooth projective variety over a perfect  field $K$, and
  let
  $n$ be a nonnegative integer.
  Let $A/K$ be an abelian variety defined over $K$, and let 
  \[
  \xymatrix{
   \phi:\operatorname{A}^{n+1}(X_{\bar K})\ar@{->>}[r] &A_{\bar K}(\bar K)
  }
  \]
  be  a surjective Galois-equivariant regular homomorphism.
  Then there are a finite number of  irreducible components of the   Hilbert
  scheme $\operatorname{Hilb}^{n+1}_{X/K}$  
  parameterizing codimension $n+1$ subschemes
  of $X/K$,  so that by taking a  finite product $H$ of these components, and
  then denoting by 
  $\operatorname{Alb}_{\widetilde H/K}$  the Albanese variety of  a  smooth
  alteration $\widetilde H$ of  $H$,  there is a
  surjective  morphism 
  \begin{equation}
    \label{E:albanesetheorem}
  \xymatrix{
   \operatorname{Alb}_{\widetilde H/K} \ar@{->>}[r]&  A
  }
\end{equation}
of abelian varieties over $K$.
 \end{teo}

 \begin{proof}
  Let  $Z$ be the cycle on $A\times_ K X$ from \cite[Lem.~4.9(d)]{ACMVdcg}  so
  that the composition
  $$
  \begin{CD}
  A(\bar K)@>w_{\bar Z}>> \operatorname{A}^{n+1}(X_{\bar K})@>{\phi}>>A(\bar K)
  \end{CD}
  $$
  is induced by the $K$-morphism 
  $
  r\cdot \operatorname{Id}:A\to A
  $
  for some positive integer $r$.  
  
  Now using  Bertini's theorem,  let  $C$ be a smooth projective curve that is a
  linear section of 
  $A$ passing through the origin (so it has a $K$-point),  and such that the
  inclusion $C\hookrightarrow A$ induces a surjective  morphism
  $J_{C/K}\twoheadrightarrow A$.  Denote again by $Z$ the refined Gysin 
  restriction of the
  cycle $Z$ to $C$.  We have  a commutative diagram
  \begin{equation}\label{E:HilbThm1}
  \xymatrix@R=1em{
   C(\bar K) \ar[d] \ar@{^(->}[r]& A(\bar K)\ar[r]^<>(0.5){w_{\bar Z}}  
   \ar@{->>}@/_2pc/[rr]_{r\cdot \operatorname{Id}}& \operatorname{A}^{n+1}(X_{\bar
    K}) \ar@{->>}[r]^{{\phi}} & A(\bar K)\\
   J_{C/K}(\bar K)  \ar@{->>}[ru]&
  }
  \end{equation}
  
  Discarding extra components, we may assume that $Z$ is flat
  over $C$.   
  Write $Z=\sum_{j=1}^mV^{(j)}-\sum_{j=m+1}^{m'}V^{(j)}$, where 
  $V^{(1)},\dots,V^{(m')}$ are (not necessarily distinct) 
   integral components of the support of $Z$, which
  by assumption are flat over $C$.  
  Let $\operatorname{Hilb}^{(j)}_{X/K}$ be the component of the Hilbert scheme,
  with universal subscheme $U^{(j)}\subseteq
  \operatorname{Hilb}^{(j)}_{X/K}\times_K X$ such that $V^{(j)}$ is obtained by
  pull-back via a morphism $f^{(j)}:C\to \operatorname{Hilb}^{(j)}_{X/K}$.  Let
  $H=\prod_{j=1}^{m'}\operatorname{Hilb}^{(j)}_{X/K}$, and let $U_H:=\sum_{j=1}^m
  \operatorname{pr}_j^*U^{(j)}-\sum_{j=m+1}^{m'}\operatorname{pr}_j^*U^{(j)}$, where $\operatorname{pr}_j : H \to \operatorname{Hilb}^{(j)}_{X/K}$ is the natural projection. 
  There is an induced morphism $f:C\to H$ and we have $Z=f^*U_H$\,; the pull-back
  is defined since all the cycles are flat over the base.    
  
  Now let $\nu:\widetilde H\to H$ be a smooth alteration of  $H$ and let
  $\widetilde U=\nu^*U_H$.  Let $\mu:\widetilde C\to C$ be an alteration such that
  there is a commutative diagram
  $$
  \xymatrix{
   \widetilde C \ar[r]^{\tilde f}  \ar[d]_\mu&  \ar[d]_\nu \widetilde H\\
   C \ar[r]^f& H.
  }
  $$
  Let $\widetilde Z=\mu^*Z$.   
  We obtain maps  
  $$
  \xymatrix{
   (\widetilde C\times_{ K} \widetilde C)(\bar K) \ar@{^(->}[r]& (\widetilde
   H\times_{K}\widetilde H)(\bar K) \ar[r]^{y_{\widetilde U_{\bar K}}} &
   \operatorname{A}^{n+1}(X_{\bar K}) \ar@{->>}[r]^\phi& A(\bar K). 
  }
  $$
  By Lemma \ref{L:DiffDesc},
  these descend to $K$-morphisms
  $$
  \xymatrix{
   \widetilde C\times_{ K} \widetilde C\ar@{^(->}[r]& \widetilde
   H\times_{K}\widetilde H \ar[r]^<>(0.5){{\xi}_{\widetilde U}} 
   & A. 
  }
  $$
Recall that if $W/K$ is any variety, then there exist an abelian
variety $\alb_{W/K}$ and a torsor $\alb^1_{W/K}$ under $\alb_{W/K}$,
equipped with a morphism $W \to \alb^1_{W/K}$ which is universal for
morphisms from $W$ to abelian torsors.
  Taking Albanese torsors we obtain a commutative diagram
  $$
  \xymatrix@R=1em{
   \widetilde C\times_{ K} \widetilde C\ar@{^(->}[r]  \ar[d] \ar@/_5pc/[ddd]&
   \widetilde H\times_{K}\widetilde H \ar[r]^<>(0.5){{\xi}_{\widetilde
     U}}   \ar[d]& A \ar@{=}[ddd]\\
   \operatorname{Alb}^1_{\widetilde C/K}\times_{ K} \operatorname{Alb}^1_{\widetilde
    C/K} \ar@{->}[r] \ar@{->>}[d] & \operatorname{Alb}^1_{\widetilde
    H/K}\times_{K}\operatorname{Alb}^1_{\widetilde H/K} \ar[ru]   & \\
   J_{ C/K}\times_{ K} J_{C/K}\ar@{->>}[rrd]& &\\
   C\times_{ K} C \ar[u] \ar@{->}[rr]_<>(0.5){{\xi}_{Z}}&    & A.\\   
  }
  $$
  The surjectivity of the map $J_{ C/K}\times_{ K} J_{C/K}\to A$  follows from
  \eqref{E:HilbThm1}.  A diagram chase then shows that the map
  $\operatorname{Alb}^1_{\widetilde H/K}\times_{ K} \operatorname{Alb}^1_{\widetilde
    H/K}\to A$ is surjective.  In general, if $T$ is a torsor under an
  abelian variety $B/K$, and if $T \twoheadrightarrow  A'$ is a surjection to an abelian variety, then there is
  a surjection $B \to A'$ over $K$.  (Indeed, the surjection $T \twoheadrightarrow  A'$
  induces an inclusion $\pic^0_{A'/K} \hookrightarrow \pic^0
_{T/K}$\,; but
  $\pic^0_{A'/K}$ is isogenous to $A'$, while $\pic^0_{T/K}$ is isogenous to $B$.)
 Applying this to the surjection   $\operatorname{Alb}^1_{\widetilde H/K}\times_{ K} \operatorname{Alb}^1_{\widetilde
    H/K}\twoheadrightarrow A$, we obtain the surjection   $\operatorname{Alb}_{\widetilde H/K}\times_{ K} \operatorname{Alb}_{\widetilde
    H/K}\to A$.
Theorem \ref{T:MazQ1} now follows, where the $\til H$ in
  \eqref{E:albanesetheorem} is the product $\til H \times_K \til H$
  considered here.
  \end{proof}

 We now use  Theorem \ref{T:MazQ1} to prove Theorem \ref{Ta:MazQ1}.
 
 \begin{proof}[Proof of Theorem \ref{Ta:MazQ1}] 
 	Recall the fundamental result of Griffiths \cite[p.~826]{griffiths68} asserting that  the Abel--Jacobi map  ${AJ}:\operatorname{A}^{n+1}(X_{\cx})
 	\longrightarrow J_a^{2n+1}(X_\cx)$ is a surjective regular homomorphism. By Theorem \ref{T:main} and its proof, $J_a^{2n+1}(X_\cx)$ descends uniquely to an abelian variety $J/K$ such that the surjective regular homomorphism $\LKtrace{AJ}:\operatorname{A}^{n+1}(X_{\bar K}) \longrightarrow J_{\bar K}$ defined in the proof of Lemma \ref{L:T-G-AC} is Galois-equivariant.
 	Now employ Theorem \ref{T:MazQ1}.
 \end{proof}

 \begin{proof}[Proof of Corollary \ref{Ca:MazQ1uni}]
  A uniruled threefold has Chow group of zero-cycles supported on a surface.  A
  decomposition of the diagonal argument \cite{BlSr83} shows that the threefold has geometric
  coniveau $1$ in degree $3$.  
 \end{proof}
 
 Theorem \ref{T:MazQ1} also gives the following result for   algebraic
 representatives.

 \begin{cor}\label{C:ChowAlgRep}
  Let $X$ be a smooth projective variety over a perfect field $K$, let
  $\Omega/\bar  K$ be an algebraically closed field extension, with either
  $\Omega=\bar K$ or $\operatorname{char}(K)=0$,  and let $n$ be a nonnegative 
  integer.    
  Assume  there is an algebraic representative
  $\phi^{n+1}_{\Omega}:\operatorname{A}^{n+1}(X_{\Omega})\to
  \operatorname{Ab}^{n+1}(X_{\Omega})(\Omega)$ (e.g., $n=0$, $1$, or $\dim X-1$).

  Then the abelian variety $\operatorname{Ab}^{n+1}(X_{ \Omega})$ descends to an
  abelian variety $\underline{\operatorname{Ab}}^{n+1}(X_{\bar K})$ over $K$, and 
  there are a finite number of  irreducible components of the   Hilbert scheme
  $\operatorname{Hilb}^{n+1}_{X/K}$ parameterizing codimension $n+1$ subschemes
  of $X/K$,  so that by taking a  finite product $H$ of these components, and
  then denoting by 
  $\operatorname{Alb}_{\widetilde H/K}$  the Albanese of  a  smooth
  alteration $\widetilde H$ of  $H$,  there is a
  surjective  morphism 
  $
  \operatorname{Alb}_{\widetilde H/K}\twoheadrightarrow
  \underline{\operatorname{Ab}}^{n+1}(X_{\bar K})
  $
  of abelian varieties over $K$.
 \end{cor}

 \begin{proof}
  The fact that $\operatorname{Ab}^{n+1}(X_{ \Omega})$ descends to $\bar K$ to
  give $\operatorname{Ab}^{n+1}(X_{\bar K})$ is \cite[Thm.~3.7]{ACMVdcg}.   It is
  then shown in \cite[Thm.~4.4]{ACMVdcg} that  $\operatorname{Ab}^{n+1}(X_ {\bar
   K})$ descends to an abelian variety over $K$ and that the map $\phi^{n+1}_{\bar
   K}:\operatorname{A}^{n+1}(X_{\bar K })\to
  \operatorname{Ab}^{n+1}(X_{\bar K})(\bar K)$ is
  $\operatorname{Gal}(K)$-equivariant. Therefore,  we may employ Theorem
  \ref{T:MazQ1} to conclude. 
 \end{proof}

 \appendix

 \section{Cohomology of Jacobians of curves via Abel maps}\label{A:Appendix}

 Let $C$ be a smooth projective curve over a field $K$ with separable
 closure $\bar K$. 
  For any $n$ invertible
 in $K$, the Kummer sequence of \'etale sheaves on $C$\,:
 \[
 \xymatrix{
  1 \ar[r] & \mmu_n \ar[r] & \gp_m \ar[r]^{[n]} & \gp_m \ar[r] & 1
 }
 \]
 gives an isomorphism
 \[
 H^1(\bar C,\mmu_n) \iso\pic_{C/K}[n] = \pic^\circ_{C/K}[n],
 \]
 where we have written $\bar C$ for $C_{\bar K}$.
 After taking the inverse limit over all powers of a fixed prime $n=\ell$,
 we obtain isomorphisms of $\gal(K)$-representations
 \[
 H^1(\bar C, \integ_\ell(1)) \iso T_\ell \pic^\circ_{C/K} \iso (T_\ell \widehat{\pic^\circ_{C/K}})^\vee(1) \iso H^1(\widehat{\pic^\circ_{C/K,\bar K}},\integ_\ell(1)).
 \]
After twisting by $-1$, the canonical
 (principal) polarization on the Jacobian gives an isomorphism
 \begin{equation}
 \label{E:H1isom}
 H^1(\bar C,\integ_\ell) \iso H^1(\pic^\circ_{\bar C/\bar K},
 \integ_\ell).
 \end{equation}
 In this appendix we show that, up to tensoring with $\rat_\ell$,  the
 isomorphism \eqref{E:H1isom} is induced by a $K$-morphism
 $C \ra \pic^\circ_{C/K}$.

 \begin{pro}
  \label{P:curveCoh}
  Let $C$ be a smooth projective curve over a field $K$.
  Then there exists a morphism $\beta:C \to \pic^\circ_{C/K}$ over $K$ which
  induces an isomorphism
  \[
  \xymatrix{
   \beta^*:H^1(\pic^\circ_{\bar C/\bar K},\integ_\ell) \ar[r]^<>(0.5)\sim &
   H^1(\bar C,\integ_\ell)
  }
  \]
  of $\gal(K)$-representations for all but finitely many $\ell$.  For
  all $\ell$ invertible in $K$, we have  that the pull-back
  $\beta^*:H^1(\pic^\circ_{\bar C/\bar K},\rat_\ell) \to H^1(\bar C,\rat_\ell)$ is
  an isomorphism.
 \end{pro}
 
 The case of an integral curve over an algebraically closed field is
 standard (e.g., \cite[Prop.~9.1,p.113]{milneAV}).  
 The case where $C$ is geometrically irreducible and
 $C(K)$ is nonempty is certainly well-known\,; even if $C$ admits no
 $K$-points, the result follows almost immediately from the case $K=\bar K$\,:
 
 \begin{lem}
  \label{L:appspecialcase}
  If $C/K$ is geometrically irreducible, then Proposition
  \ref{P:curveCoh} holds for $C$.
 \end{lem}

 \begin{proof}
  Let $d$ be a positive integer such that $C$ admits a line bundle
  $ L$
  of degree $d$ over $K$.  Let $\beta$ denote the composition
  \[
  \xymatrix@C=3.5em{
   \beta:C \ar[r]^<>(0.5)a & \pic^1_{C/K} \ar[r]^{[d] = (-)^{\tensor
     d}}_{\text{isogeny}} & \pic^d_{C/K} \ar[r]^{(-)\tensor
    { L}^\vee}_{\sim} & \pic^\circ_{C/K},
  }
  \]
  where $\pic^e_{C/K}$ denotes  the 
  torsor under  $\pic^\circ_{C/K}$ consisting of degree $e$ line bundles on $C/K$,
  and 
  $a$ is the Abel map (e.g., \cite[Def.~9.4.6, Rem.~9.3.9]{kleimanPIC}).  
  
  We claim that if $\ell \nmid d$, then $\beta^*: H^1(\pic^\circ_{\bar C/\bar
   K},\integ_\ell) \to H^1(\bar C,\integ_\ell)$ is an
  isomorphism.  After passage to $\bar K$, we may find a line bundle
  ${M}$
  such that ${ M}^{\tensor d} \iso   L$.  We have a commutative diagram
  \[
  \xymatrix{
   \bar C \ar[r]^<>(0.5)a \ar[dr]_{a_{ M}}& \pic^1_{\bar C/\bar K} \ar[r]^{[d]}
   \ar[d]^{(-)\tensor { M}^\vee}&
   \pic^d_{\bar C/\bar K}\ar[d]^{(-)\tensor {  L}^\vee} \\
   & \pic^\circ_{\bar C/\bar K} \ar[r]^{[d]} & \pic^\circ_{\bar C/\bar K}
  }
  \]
  Since the diagonal arrow is the usual Abel--Jacobi embedding of $\bar C$ in its
  Jacobian, where the assertion about pull back of cohomology is well known (e.g.,
  \cite[Prop.~9.1, p.113]{milneAV}),  and the lower horizontal arrow is an isogeny
  of degree $d^{2g(C)}$, the commutativity of the diagram implies   that  $\beta$
  has the asserted properties.
 \end{proof}

 \subsection{Components of the Picard scheme}
 
 Now suppose that $C$ is irreducible but $\bar C$ is reducible.
 Continue to let 
 $\pic^\circ_{\bar C / \bar K}$ denote the connected component of
 identity of the Picard scheme, and for each $d$ let $\pic^d_{\bar
  C/\bar K}$ be the space of line bundles of total degree $d$.  (This
 has the unfortunate notational side effect that $\pic^\circ_{\bar C/\bar K}$
 does
 not coincide with $\pic^0_{\bar C/\bar K}$, but we will never have
 cause to study the space of line bundles of total degree zero.)  Then
 $\pic^d_{\bar C/\bar K}$ is no longer a
 torsor under $\pic^\circ_{\bar C/\bar K}$, and we need to work slightly
 harder to identify suitable geometrically irreducible, $K$-rational
 components of the Picard scheme of $C$.
 
 Let $\Pi_0(\bar C)$ be the set of irreducible
 components of $\bar C$.  (Since $\bar K$ is separably closed, each
 such component is geometrically irreducible.)  Fix a component $\bar D\in \Pi_0(C_{\bar
  K})$, and let $H\subset \gal(K)$ be its stabilizer.  Since $C$ is
 irreducible, we have
 \[
 \bar C = \bigsqcup_{[\sigma] \in \gal(K)/H} \bar D^\sigma,
 \]
 where viewing $\sigma$ as an automorphism of $\bar C$, we set $\bar D^\sigma
 =\sigma(\bar D)$.  
 Let $e = \#\Pi_0(\bar C)$.  Inside the $de$-th symmetric power
 $S^{de}(C)_{\bar K} = S^{de}(\bar C)$ we identify the irreducible
 component
 \[
 S^{\Delta_d}(\bar C) := \prod_{[\sigma] \in \gal(K)/H} S^d(\bar
 D^\sigma).
 \]
 Since this element of $\Pi_0(S^{de}(\bar C))$ is fixed by
 $\gal(K)$, it descends to $K$ as a geometrically irreducible variety.
 
 Similarly, inside the Picard scheme $\pic_{\bar C/\bar K}$ we single out
 \[
 \pic^{\Delta_d}_{\bar C/\bar K} = \prod_{[\sigma] \in \gal(K)/H}
 \pic^d_{\bar D^\sigma/\bar K}.
 \]
 It is visibly irreducible and, since it is stable under $\gal(K)$, it
 descends to $K$. Note that $\pic^{\Delta_d}_{\bar C/\bar K}$ is a
 $\pic^{\circ}_{\bar C/\bar K}$-torsor.
 
 The $(de)$-th Abel map $S^{de}(C) \to \pic^{de}_{C/K}$ then restricts to a
 morphism
 \[
 \xymatrix{
  S^{\Delta_d}(C) \ar[r]^{a_{\Delta_d}} & \pic^{\Delta_d}_{C/K}
 }
 \]
 of geometrically irreducible varieties over $K$.
 
 One (still) has the canonical Abel map
 \[
 \xymatrix{
  C \ar[r]^<>(0.5)a & \pic^1_{C/K}.
 }
 \]
 Over $\bar K$, the image of the Abel map $a_{\bar K}$ lands in 
 \[
 \pic^\one_{\bar C/\bar K} = \bigsqcup_{[\sigma] \in \gal(K)/H} \left(
 \pic^1_{\bar D^\sigma} \times \prod_{[\tau]\not = [\sigma]} \pic^\circ_{\bar
  D^\tau} \right).
 \]
 Although $\pic^\one_{\bar C/\bar K}$ has $e$ components, $\gal(K)$ acts
 transitively on them, and we have an irreducible variety $\pic^\one_{C/K}$
 over $K$.

 In conclusion, the canonical Abel map induces a morphism
 \[
 \xymatrix{
  C \ar[r]^<>(0.5)a & \pic^\one_{C/K}
 }
 \]
 of irreducible varieties over $K$.
 
 We need two more $K$-rational morphisms\,:
 
 \begin{lem}
  Let $C/K$ be a smooth projective integral curve.  
  Let $s$ be the map
  \[
  \xymatrix@R=.5em{
   \pic^\one_{\bar C/\bar K} \ar[r]^s & \pic^{\Delta_1}_{\bar C/\bar K} \\
   L \ar@{|->}[r] & \bigotimes_{[\sigma] \in \gal(K)/H}
   \sigma^*{ L}.
  }
  \]
  Let $t$ be the map
  \[
  \xymatrix{
   \bar C \ar[r]^<>(0.5)t & S^{\Delta_1}(\bar C)
  }
  \]
  such that, if $P \in \bar D^\tau(K) \subset C(\bar K)$, then the
  components of $t(P)$ are given by
  \[
  t(P)_\sigma = \sigma\tau\inv(P) \in \bar D^\sigma
  \]
  Then $s$ and $t$ descend to morphisms over $K$.
 \end{lem}
 
 \begin{proof}
  Each is $\gal(K)$-equivariant on $\bar K$-points.
 \end{proof}
 
 \subsection{Isomorphisms on cohomology}
 
 \begin{lem}
  \label{L:sona}
  Let $C/K$ be a smooth projective irreducible curve.  Then the
  composition 
  \[
  \xymatrix{
   C \ar[r]^<>(0.5)a & \pic^\one_{C/K} \ar[r]^s & \pic^{\Delta_1}_{C/K}
  }
  \]
  induces an isomorphism of $\gal(K)$-representations
  \[
  H^1(\pic^{\Delta_1}_{\bar C/\bar K}\integ_\ell) \to
  H^1(\bar C,\integ_\ell).
  \]
 \end{lem}
 
 \begin{proof} 
  It suffices to analyze $s \circ a$ after base change to $\bar K$.
  Choose a base point $P_\sigma \in \bar D^\sigma$ for each irreducible
  component of $\bar C$.  We have a commutative diagram
  \[
  \xymatrix{
   \bar C \ar[r]^a \ar[d]^t & \pic^{\one}_{\bar C/\bar K} \ar[r]^s &
   \pic^{\Delta_1}_{\bar C/\bar K} \ar[d]^{\prod_{[\sigma]} (-)\tensor
    {\mathcal O}(-P_\sigma)} \\
   S^{\Delta_1}(\bar C) \ar[rru]^{a_{\Delta_1}}
   \ar[rr]_{\prod_{[\sigma]}a_{P_\sigma}}
   & &\pic^\circ_{\bar C/\bar K},
  }
  \]
  where the bottom arrow is the product of Abel maps associated to the points
  $P_\sigma$.  
  Since the right-most vertical arrow is an isomorphism of schemes, it
  suffices to verify that $t$ and $\prod a_{P_\sigma}$ induce
  isomorphisms on first cohomology groups.  On one hand, since
  cohomology takes coproducts to products, we have $H^1(\bar C,\integ_\ell) \iso
  \prod_\sigma H^1(\bar D^\sigma,\integ_\ell)$.
  On the other hand, since each $\bar D^\sigma$
  is connected, the K\"unneth formula implies that
  $H^1(S^{\Delta_1}(\bar C), \integ_\ell) =H^1(\prod_\sigma \bar D^\sigma,
  \integ_\ell) \iso \oplus_\sigma 
  \operatorname{pr}_\sigma^* H^1(\bar D^{\sigma}, \integ_\ell)$.  Since the
  composition 
  $
  \xymatrix{
   \bar D^\tau \ar[r]^<>(0.5)t & \prod_\sigma \bar D^\sigma
   \ar[r]^<>(0.5){\operatorname{pr}_\tau} & \bar D^\tau
  }
  $
  is the identity,
  \[
  \xymatrix{
   H^1(S^{\Delta_1}(\bar C), \integ_\ell) \ar[r]^<>(0.5){t^*} &
   H^1(\bar C, \integ_\ell)
  }
  \]
  is an isomorphism as well.

  Finally, since each Abel--Jacobi map $a_{P_\sigma}$ induces an isomorphism
  $H^1(\pic^\circ_{\bar D^\sigma/\bar K},\integ_\ell) \iso H^1(\bar D^\sigma,
  \integ_\ell)$, their product yields an isomorphism 
  $(\prod_{[\sigma]}a_{P_\sigma})^*:H^1(\pic^\circ_{\bar C/\bar K},\mathbb Z_\ell)
  \to H^1(S^{\Delta_1}(\bar C), \integ_\ell)$.
 \end{proof}
 
 It is now straight-forward to provide a proof of the main result of
 this appendix.

 \begin{proof}[Proof of Proposition \ref{P:curveCoh}]
  Since both the Picard functor and cohomology take coproducts to
  products, we may and do assume that $C$ is irreducible.
  Choose $d$ such that $\pic^{\Delta_d}_{C/K}$ admits a $K$-point $L$.  Let
  $\beta$ be the composition
  \[
  \xymatrix{
   C \ar[r]^<>(0.5)a & \pic^\one_{C/K} \ar[r]^s & \pic^{\Delta_1}_{C/K}
   \ar[r]^{[d]}_{\operatorname{isog.}} & \pic^{\Delta_d}_{C/K} \ar[r]^{(-)\tensor
    L^\vee}_\cong &
   \pic^\circ_{C/K}.
  }
  \]
  By Lemma \ref{L:sona}, $\beta^*: H^1(\pic^\circ_{\bar C/\bar K},\integ_\ell)
  \to H^1(\bar C,\integ_\ell)$ is an isomorphism as long as $\ell
  \nmid d$.
 \end{proof}

 \bibliographystyle{amsalpha}
 \bibliography{DCG}
\end{document}